\documentclass[onefignum,onetabnum]{siamart190516}
\hypersetup{colorlinks=true,linkcolor=blue,citecolor=blue}

\usepackage{setspace,url}
\usepackage[algo2e,linesnumbered,vlined,ruled]{algorithm2e}
\usepackage{graphics,graphicx,epsf,epstopdf,subfigure}
\usepackage{comment}
\usepackage{amsmath,amsxtra,amsfonts,amscd,amssymb,bm}
\usepackage{supertabular,multirow}
\usepackage{booktabs}
\usepackage[normalem]{ulem}
\usepackage[title]{appendix}
\usepackage{exscale}
\usepackage[margin=1in,headheight=13.6pt]{geometry}
\usepackage{dcolumn}
\usepackage{comment}

\usepackage{tikz,tikz-3dplot}
\usetikzlibrary{calc,positioning,intersections}
\tdplotsetmaincoords{70}{120}
\usetikzlibrary{patterns}

\newtheorem{remark}{Remark}

\usepackage{changes}
\definechangesauthor[name={Yang}, color=orange]{Yang}

\definechangesauthor[name={ZC}, color=orange]{ZC}

\definecolor{kyred}{rgb}{0.,0.,0.}

\DeclareMathOperator\supp{supp}
\DeclareMathOperator*\argmin{argmin}

\usepackage{upgreek}

\makeatletter
\def\namedlabel#1#2{\begingroup
    #2%
    \def\@currentlabel{#2}%
    \phantomsection\label{#1}\endgroup
}
\makeatother

\begin{document}
\title{On the validity of complex Langevin method for path integral computations\thanks{Submitted to the editors July 20, 2020. The authors would like to thank Dr. Lei Zhang at National University of Singapore for useful discussions.\funding{Zhenning Cai and Yang Kuang were supported by the Academic Research Fund of the Ministry of Education of Singapore under grant No. R-146-000-291-114.}}}
\headers{On the validity of CL method}{Zhenning Cai, Xiaoyu Dong, and Yang Kuang}

\author{
    Zhenning Cai\thanks{Department of Mathematics, National University of Singapore, Singapore 119076 (\email{matcz@nus.edu.sg}).}
    \and
    Xiaoyu Dong\thanks{LSEC, Academy of Mathematics and Systems Science, Chinese Academy of Sciences, Beijing 100190, China (\email{dongxiaoyu@lsec.cc.ac.cn}).}
	\and
	Yang Kuang\thanks{Department of Mathematics, National University of Singapore, Singapore 119076 (\email{matkuan@nus.edu.sg}).}
}

\date{\today}
\maketitle

\begin{abstract}
  The complex Langevin (CL) method is a classical numerical strategy to alleviate the numerical sign problem in the computation of lattice field theories. Mathematically, it is a simple numerical tool to compute a wide class of high-dimensional and oscillatory integrals. However, it is often observed that the CL method converges but the limiting result is incorrect. The literature has several unclear or even conflicting statements, making the method look mysterious. By an in-depth analysis of a model problem, we reveal the mechanism of how the CL result turns biased as the parameter changes, and it is demonstrated that such a transition is difficult to capture. Our analysis also shows that the method works for any observables only if the probability density function generated by the CL process is localized. To generalize such observations to lattice field theories, we formulate the CL method on general groups using rigorous mathematical languages for the first time, and we demonstrate that such localized probability density function does not exist in the simulation of lattice field theories for general compact groups, which explains the unstable behavior of the CL method. Fortunately, we also find that the gauge cooling technique creates additional velocity that helps confine the samples, so that we can still see localized probability density functions in certain cases, as significantly broadens the application of the CL method. The limitations of gauge cooling are also discussed. In particular, we prove that gauge cooling has no effect for Abelian groups, and we provide an example showing that biased results still exist when gauge cooling is insufficient to confine the probability density function.
\end{abstract}


\section{Introduction}\label{sec:introduction}
Quantum field theory (QFT) is a fundamental theoretical framework in particle physics and condensed matter physics, which has achieved great success in explaining and discovering elementary particles in the history. Although QFT still lacks a rigorous mathematical foundation, there have already been numerous approaches to carrying out computations in QFT, based on either perturbative or non-perturbative approaches. Perturbative approaches can be applied when the coupling constant, which appears in the coupling term in the Lagrangian describing the interaction between particles, is relatively small, so that the asymptotic expansions with respect to the coupling constant, often denoted by Feynman diagrams, can be adopted as approximations. In quantum chromodynamics (QCD), which studies the interaction between quarks and gluons, the perturbative approaches work in the case of large momentum transfers. However, when studying QCD at small momenta or energies (less than 1GeV), due to renormalization, the coupling constant is comparable to 1 and the perturbative theory is no longer accurate \cite{greiner2007quantum}. Therefore, one has to resort to non-perturbative approaches, typically lattice QCD calculations, to obtain reliable approximations of the observables.

Lattice QCD is formulated based on the path-integral quantization of the classical gauge field theory. In general, the expectation of any observable $O$ can be computed by evaluating the discrete path integral
\begin{equation*}
\langle O \rangle = \frac{1}{Z} \int_{\Omega} O(x) e^{-S(x)} \,dx, \qquad
  Z = \int_{\Omega} e^{-S(x)} \,dx,
\end{equation*}
where $\Omega$ is a space whose number of dimensions is often larger than $10^4$, and $S(x)$ is the action. This problem is of broad interest in statistical mechanics \cite{gibbs2010elementary}, quantum mechanics \cite{cohen2015taming} and string theories \cite{anagnostopoulos2002new}. Due to the high dimensionality, one has to apply Monte Carlo methods to compute this integral. However, when the chemical potential is nonzero, the action is no longer real-valued, so that $e^{-S(x)}$ is no longer positive, which then causes severe \emph{numerical sign problem} in the computation \cite{gattringer2010quantum}. Here numerical sign problem refers to the phenomenon that the variance of a stochastic quantity is way larger than its expectation, resulting in significant difficulty in reducing the relative error in Monte Carlo simulations. The large variance is usually due to strong oscillation of the stochastic quantity, causing significant cancellation of its positive and negative contributions. Such problem typically occurs in quantum Monte Carlo simulations, including both lattice field theory \cite{loh1990sign} and real-time dynamics \cite{schiro2010real}. In lattice QCD, the imaginary part of $S(x)$ contributes to the high oscillation, so that the ``partition function'' $Z$ itself already has a small value and is difficult to compute accurately.

In general, there is no universal approach to solving the numerical sign problem. For example, the inchworm Monte Carlo method \cite{cohen2015taming, chen2017inchworm}, which takes the idea of partial resummation, has been proposed to mitigate the numerical sign problem for real-time dynamics of the impurity model or open quantum systems; the Lefschetz thimble method \cite{cristoforetti2012new, cristoforetti2013monte}, which uses Morse theory to change the integral path, has been applied in lattice QCD computations. In this work, we are interested in another approach to taming the numerical sign problem, known as the complex Langevin (CL) method, whose basic idea is to search for a positive probability function in a higher-dimensional space that is equivalent to the ``complex probability density function'' $Z^{-1} e^{-S(x)}$, so that we can apply the Monte Carlo method in this higher-dimensional space, which is free of numerical sign problems \cite{seiler2018beyond}. In the CL method, such a space is established by complexifying all the variables, meaning to replace all real variables by complex variables, and extending all functions defined on $\mathbb{R}$ to functions on $\mathbb{C}$ by analytic continuation. Thus the number of dimensions is doubled. The equivalent probability density function in this complexified space is sampled by complexifying the Langevin method. Such a method is also known as stochastic quantization \cite{namiki1992stochastic}. We refer the readers to \cite{seiler2018status} for a recent review of this method.

Since the CL method was proposed in \cite{klauder1983langevin, parisi1983complex}, it has never been fully understood theoretically. In numerical experiments, it is frequently seen that this method generates biased results, and its validity has therefore been questioned by a number of researchers \cite{gausterer1998complex, aarts2010complex, salcedo2016does}. Its application had been rarely seen until one breakthrough of the CL method, called the gauge cooling technique, was introduced in \cite{seiler2013gauge}. Such strategy utilizes the redundant degrees of freedom in the gauge field theory to stabilize the method. With this improvement, the CL method has been successfully applied to finite density QCD \cite{aarts2014simulating, sexty2014simulating, Kogut2019applying} and the computations in the superstring theory \cite{anagnostopoulos2018complex, nishimura2019complex}. Recently, it has also been used in the computation of spin-orbit coupling \cite{attanasio2020thermodynamics}. Despite the success, failure of the CL method still occurs, and researchers are still working hard on understanding the algorithm by formal analysis and some particular examples \cite{aarts2013localised, nagata2016argument, salcedo2016does}, hoping to make further improvements. For instance, the recent work \cite{scherzer2019complex} analyzes a specific one-dimensional case, where the authors quantified the bias by relating it to the boundary terms when performing integration by parts. The results therein have been further applied in \cite{scherzer2020controlling} to correct the CL results. {\color{kyred}In \cite{aarts2017complex}, the authors analyze the role of poles in the Langevin drift and find that although the converged CL results satisfy the Schwinger-Dyson equation, the integrals may still be incorrectly predicted. This phenomenon is further studied in \cite{salcedo2018schwinger} in a rigorous way for the one-dimensional case.} The issue revealed in \cite{scherzer2019complex} shows that the correctness of the CL method depends highly on the decay rate of the probability density function and the growth rate of the observable at infinity, which is difficult to predict before the computation, especially in the high-dimensional case. The only case in which the convergence can be guaranteed for any observable is when the probability density function is localized, which has been studied in \cite{aarts2013localised} for a one-dimensional model problem.

In this paper, we will carry out a deeper study of this specific case, and clarify some unclear statements and some misunderstanding in previous studies. Later, we will show that in lattice QCD simulations, the localized probability density function can occur only after gauge cooling is introduced, although it is not always effective. This reveals why this technique is essential to the CL method.

To begin with, we will provide a brief review of the CL method, and introduce the model problem studied in \cite{aarts2013localised}.

\subsection{Review of the complex Langevin method} \label{sec:review}
To sketch the basic idea of the complex Langevin method in the lattice field theory, we can simply consider the one-dimensional integration problem which aims to find
\begin{equation}\label{eq:integral}
  \left\langle O \right\rangle = \frac{1}{Z} \int_{\mathbb{R}}
  O(x)e^{-S(x)} \,dx, \quad Z = \int_{\mathbb{R}} e^{-S(x)} \,dx.
\end{equation}
When $S(x)$ is real, we can regard $Z$ as the partition function, so that the integral can be evaluated by the Langevin method. Specifically, we can approximate \eqref{eq:integral} by
\begin{displaymath}
\langle O \rangle \approx \frac{1}{N} \sum_{k=1}^N O(X_k),
\end{displaymath}
where $X_k$ are samples generated by simulating the following Langevin equation:
\begin{equation} \label{eq:real_Langevin}
dx = K(x) \,dt + dw, \quad K(x) = -S'(x),
\end{equation}
and we choose $X_k = x(T + k \Delta T)$ for a sufficiently large $T$ and sufficiently long time difference $\Delta T$. In \eqref{eq:real_Langevin}, $w(t)$ is the standard Brownian motion satisfying $dw^2 = 2dt$. The method converges when the stochastic process is ergodic. We refer the readers to \cite{mattingly2002ergodicity} for more details about the theory of ergodicity.

When $S(x)$ is complex, the theory of the above method breaks down, since $Z$ is no longer a partition function. Interestingly, the above algorithm can still be carried out, at least formally, if the functions $O(x)$ and $S(x)$ can be extended to the complex plane. Now we assume that both $O(x)$ and $S(x)$ are analytic and can be extended to $\mathbb{C}$ holomorphically. By naming the new functions as $O(z)$ and $S(z)$, we can still carry out the above process by the replacement $x \rightarrow z$ and $X_k \rightarrow Z_k$. More specifically, if we denote $z$ by $x+iy$, $x,y \in \mathbb{R}$, then the integral \eqref{eq:integral} is approximated by
\begin{equation} \label{eq:approx_O}
\langle O \rangle \approx \frac{1}{N} \sum_{k=1}^N O(X_k + iY_k),
\end{equation}
where the samples $X_k$ and $Y_k$ are generated by simulating the following \emph{complex Langevin equation}:
\begin{equation}\label{eq:sto}
  \left\{
    \begin{array}{@{}lll}
      dx = K(x,y) \,dt + dw,& K(x,y) = \mbox{Re}(-S'(x+iy)),  \\[6pt]
      dy = J(x,y) \,dt,& J(x,y) = \mbox{Im}(-S'(x+iy)),
    \end{array}
  \right.
\end{equation}
and choosing $X_k = x(T + k \Delta T)$ and $Y_k = y(T + k \Delta T)$. For the initial condition, we require that $y(0)=0$ and $x(0)$ can be an arbitrary real number. This method is known as the \emph{complex Langevin method}.

The validity of the CL method is usually studied using the dual Fokker-Planck (FP) equation, which describes the evolution of the probability density function of $x(t)$ and $y(t)$ for the SDE \eqref{eq:sto}. Using $P(x,y;t)$ to denote the joint probability of $x(t)$ and $y(t)$, we can derive from \eqref{eq:sto} that
\begin{equation}\label{eq:fp}
  \partial_t P = L^TP, \qquad P(x,y;0) = p(x) \delta(y),
\end{equation}
where $p(x)$ is a probability density function on the real axis, and $L^T$ represents the Fokker-Planck operator:
\begin{equation}\label{eq:fpoperator}
  L^T P = \partial_{xx} P - \partial_x (K_x P) - \partial_y (K_y P).
\end{equation}
Thus the right-hand side of \eqref{eq:approx_O} converges to the quantity
\begin{equation} \label{eq:limit}
\lim_{t\rightarrow +\infty} \int_{\mathbb{R}} \int_{\mathbb{R}} O(x+iy) P(x,y;t) \,dx \,dy.
\end{equation}
The justification of the CL method requires us to check whether the above quantity equals $\langle O \rangle$.

Such equivalence has been shown in \cite{aarts2010complex, scherzer2019complex} under certain conditions. Here we would like to restate the result as a rigorous theorem, which requires the following assumptions on the observable function $O(x)$ and the ``complex probability function'':
\begin{description}
\item[\namedlabel{itm:H1}{(H1)}] Let $\mathcal{O}(x,y;t)$ be the solution of the backward Kolmogorov equation
  \begin{equation*}
    \partial_t \mathcal{O} = L \mathcal{O},\qquad \mathcal{O}(x,y;0) = O(x+iy),
  \end{equation*}
  where $L = \partial_{xx} + K_x \partial_x + K_y \partial_y$.
  It holds that
  \begin{equation} \label{eq:OP}
    \int_{\mathbb{R}} \int_{\mathbb{R}} \mathcal{O}(x,y;\tau) P(x,y;t-\tau) \,dx \,dy =
    \int_{\mathbb{R}} \int_{\mathbb{R}} O(x+iy) P(x,y;t) \,dx \,dy
  \end{equation}
  for any $0\le\tau\le t$.
\item[\namedlabel{itm:H2}{(H2)}] The ``forward Kolmogorov equation'' for the complex-valued function $\rho(x;t)$
  \begin{equation} \label{eq:rho}
    \partial_t \rho = \partial_x (S'(x) \rho) + \partial_{xx} \rho,\qquad \rho(x;0) = p(x)
  \end{equation}
  has the {unique} steady state solution
  \begin{displaymath}
    \lim_{t\rightarrow +\infty} \rho(x;t) = \rho_{\infty}(x) = \frac{1}{Z} e^{-S(x)}.
  \end{displaymath}
\item[\namedlabel{itm:H3}{(H3)}] For any $0 \le \tau \le t$, it holds that
  \begin{equation*}
  \lim_{X \rightarrow \infty} [\mathcal{O}(X,0;\tau) \partial_x \rho(X;t-\tau) - \rho(X;t-\tau) \partial_x \mathcal{O}(X,0;\tau)] =
  \lim_{X\rightarrow \infty} S'(X) \mathcal{O}(X,0;\tau) \rho(X;t-\tau) = 0.
  \end{equation*}
\end{description}
Based on the above assumptions, the following theorem implies the validity of the CL method:
\begin{theorem} \label{thm:CL}
Assume that the conditions \ref{itm:H1} and \ref{itm:H3} hold. Then for any $t > 0$,
\begin{equation} \label{eq:obs_eq}
\int_{\mathbb{R}} \rho(x,t) O(x) \,dx =
  \int_{\mathbb{R}} \int_{\mathbb{R}} P(x,y;t) O(x+iy) \,dx \,dy.
\end{equation}
\end{theorem}

In the above theorem, we can take the limit $t \rightarrow +\infty$ on both sides of \eqref{eq:obs_eq}. By the assumption \ref{itm:H2}, one sees that \eqref{eq:limit} equals $\langle O \rangle$, which justifies the complex Langevin method. For comprehensiveness, we provide the proof of theorem in \cref{sec:proof}. The CL method computes the right-hand side of \eqref{eq:obs_eq}, which no longer includes rapidly sign-changing functions as long as the observable $O(x+iy)$ is not oscillatory. Thereby the numerical sign problem is mitigated.

{Unfortunately, this theorem is far from satisfactory since the hypotheses are difficult to justify. In fact, these hypotheses are often too ideal such that they are often violated, resulting in some mysterious behaviors of the CL method.} In applications, we often find the CL method fails to work due to divergence. Even worse, sometimes the algorithm appears to be convergent, but as the number of samples $N$ increases, the right-hand side of \eqref{eq:approx_O} does not converge to its left-hand side, leading to biased numerical result. This means that the conditions of \cref{thm:CL}, which look reasonable for the Langevin method, is often too strong in the application of the CL method. An example will be given in the next subsection.

\begin{remark}
In some literature, it is only required that {$S$ is meromorphic on $\mathbb{C}$, \textit{i.e.}, these functions may have countable poles. This usually occurs due to the multiple-valued logarithmic function, which is applied to the fermionic determinant in the action $S$. In this case, Theorem \ref{thm:CL} still holds if the poles do not appear on $\mathbb{R}$. While such a situation also has important applications, in this paper, we temporarily restrict ourselves to the simpler holomorphic case, and the difference will be revealed in Remark \ref{rem:holomorphic} at the end of Section \ref{sec:effect}.} Note that \cref{thm:CL} can be generalized to the multi-dimensional case without difficulty.
\end{remark}

\subsection{Failure of the CL method}\label{sec:intro-failcl}
The failure of the CL method has been demonstrated in a number of previous works \cite{salcedo2016does, nagata2016argument, scherzer2019complex}. Here we adopt the example used in \cite{aarts2013localised, nagata2016argument} to demonstrate such a phenomenon. We choose the observable as $O(x) = x^2$ and the complex action as
\begin{equation}\label{eq:toy1action}
  S(x) = \frac{1}{2}(1+iB) x^2 + \frac{1}{4} x^4,
  \quad B \in \mathbb{R}.
\end{equation}
It has been found in \cite{nagata2016argument} that the CL method converges to the correct expectation values when $B$ is small, while for large $B$, the CL method may still converge, but the limiting value differs from the integral \eqref{eq:integral}. We have redone the numerical experiment by simulating the stochastic equation \eqref{eq:sto} for $B = 1$ to $5$, and the results are given in \cref{fig:toy1cl_cl}, showing the same behavior as in \cite{nagata2016argument}. Furthermore, we find that for $B > 4$, the simulation becomes unstable in the sense that arithmetic overflow often appears, and for stable simulations, the CL result deviates from the exact integral significantly. To better confirm the phenomenon, we solve the FP equation \eqref{eq:fp} numerically using the method to be introduced in \cref{sec:toy1-nm} (see \cref{fig:toy1cl_fp}). The disagreement between the two figures for large $B$ also implies the lack of reliability for the CL simulation.

\begin{figure}[htbp]
\centering
\subfigure[Results from solving the CL equation]
{\includegraphics[width=0.35\textwidth]{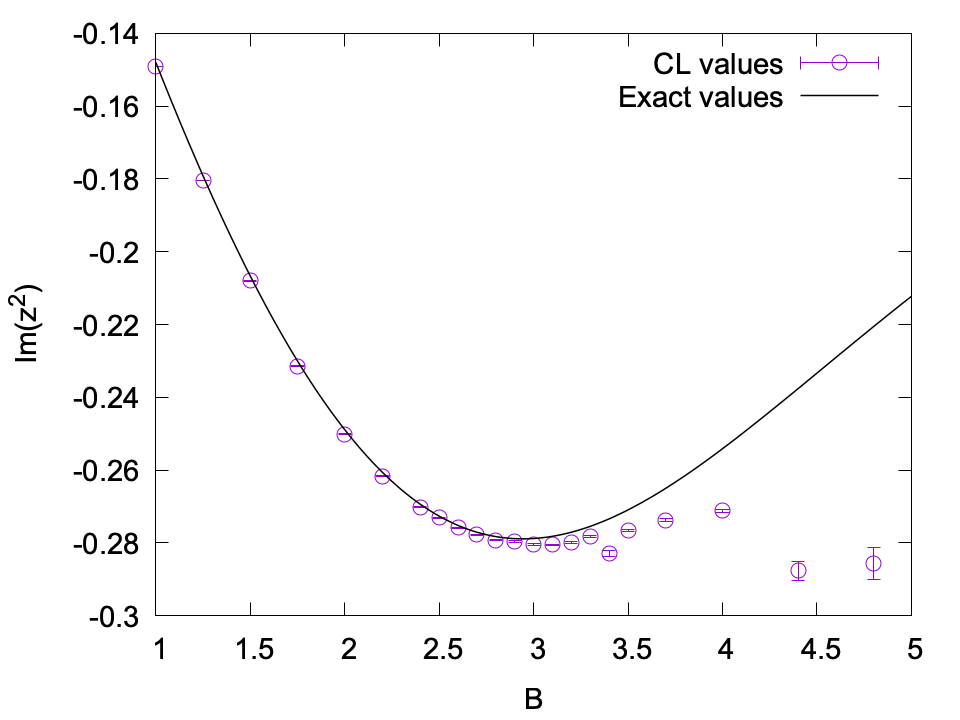}\label{fig:toy1cl_cl}}\qquad
\subfigure[Results from solving the FP equation]
{\includegraphics[width=0.35\textwidth]{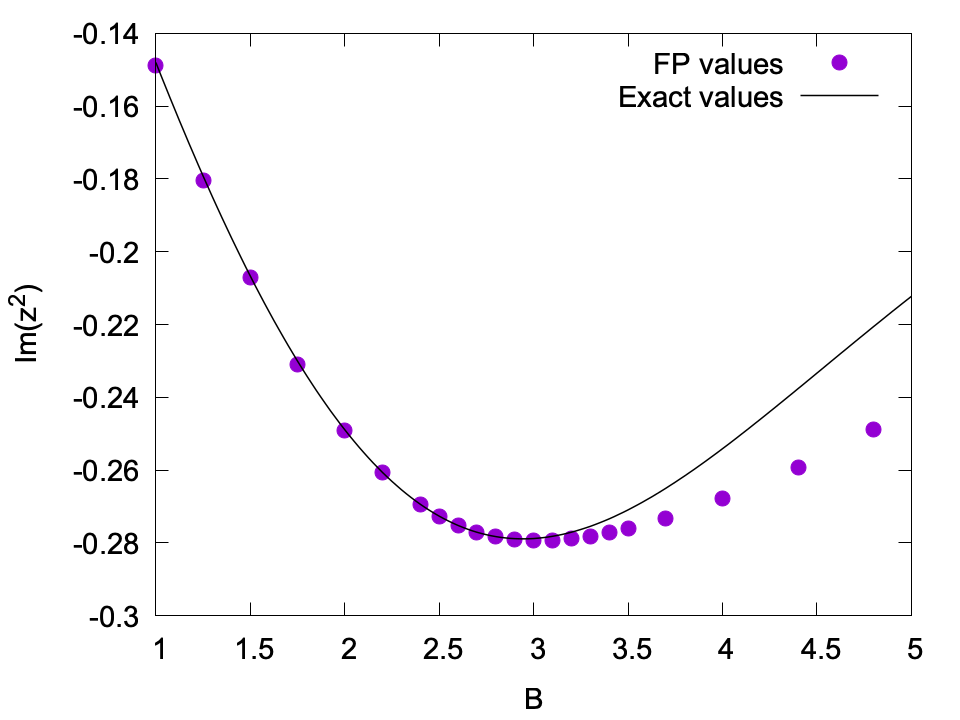}\label{fig:toy1cl_fp}}
\caption{The imaginary part of the expectation value $O(z)=z^2$.\label{fig:toy1cl}}
\end{figure}

For this specific example, it is known by the analysis in \cite{aarts2013localised} that for $B \le \sqrt{3}$, there exists $y_- > 0$ such that $P(x,y;t) = 0$ for any $x$ and $t$ if $|y| > y_-$. Furthermore, the authors of \cite{aarts2013localised} proposed the decay rate $\exp(-a x^4)$ for the marginal probability density function of $x$. The fast decay of the probability density function on the complex plane implies that the conditions \ref{itm:H1}--\ref{itm:H3} hold true, so that \cref{thm:CL} guarantees the validity of the CL method.
{\color{kyred} In this situation, following \cite{aarts2013localised}, we consider the probability density function as ``localized''. In the literature, the localization of the probability may refer to a sufficient decay of the probability density function such that the samples that drift far away from the real axis can be ignored. In this sense, as will be discussed in the next section, the admissible decay rate depends on the increasing rate of the observable, which will introduce significant difficulty to our analysis. Therefore in this paper, we only consider the localization of the probability in a strong sense, \textit{i.e.}, the probability density function is said to be localized if and only if there exists $Y \in \mathbb{R}^+$ such that $P(x,y) = 0$ for any $x \in \mathbb{R}$ and $|y| > Y$. Here we focus mainly on the localization in the $y$-direction since in the applications of the CL method (mainly QCD), the $x$-direction is usually a compact group, as will be detailed in Section \ref{sec:theory}.}

However, the behavior of the solution for $B > \sqrt{3}$ remains unclear. In \cite{aarts2013localised}, the authors predicted a power decay of the probability function when $B > \sqrt{3}$, but they left only a vague comment saying that this is ``an important signal of failure''. However, in \cite{nagata2016argument}, the authors argued by numerical experiment that the probability density function falls off exponentially for $B\leq 2.6$, implying again that \cref{thm:CL} holds and the CL method is unbiased, and when $B > 2.8$, power decay is observed, indicating the failure of the complex Langevin method.

These conflicting results for this model problem also imply the lack of understanding of the CL method. In this work, we will begin our discussion from this toy example, and try to answer the following questions:
\begin{itemize}
  \item How does the power decay of the probability affect the numerical value of the expectation? How is it related to the conditions of \cref{thm:CL}?
  \item What is the critical value of $B$ from which the CL result deviates from the exact integral?
  \item What occurs to the distribution function when $B$ passes this critical value? Is it a smooth transition?
\end{itemize}
After getting more mathematical insights of this model problem, we will further study the phenomenon of localized probability density functions in the lattice field theory. The rest of the paper is organized as follows: \cref{sec:toy_model} is devoted to a comprehensive study of the model problem \eqref{eq:toy1action}. The formulation of the CL method in lattice field theories will be provided in \cref{sec:theory}, where we will also demonstrate the non-existence of localized probability density functions. In \cref{sec:gc}, we study the effect of the gauge cooling technique in localizing the probability density functions. Finally, some concluding remarks are given in \cref{sec:conclusion}.
\section{A study of the model problem} \label{sec:toy_model}

This section is devoted to a closer look at the integration problem with action \eqref{eq:toy1action}, for which the drift velocities are
\begin{equation}
  K(x,y) = -\left(x-B y+x^{3}-3 x y^{2}\right), \qquad
  J(x,y) = -\left(y+B x+3 x^{2}
    y-y^{3}\right). \label{eq:toy1Ky} 
\end{equation}
Following \cite{nagata2016argument}, we study the complexified observable $O(z) = z^2 = x^2-y^2 +i2xy$. To understand the phenomenon showing in \cref{fig:toy1cl}, we will first resolve the ambiguity in the literature about the decay rate of the probability density function, this will be done by solving the FP equation \eqref{eq:fp} numerically.

\subsection{Numerical method for the Fokker-Planck equation}\label{sec:toy1-nm}
The FP equation for the complex Langevin equation has been numerically solved in \cite{aarts2013localised, scherzer2019complex}. It is found in \cite{scherzer2019complex} that according to the CFL condition, very small time steps need to be adopted by explicit schemes due to the large values of velocities when $|x|$ and $|y|$ are large. Therefore in our study, we take the numerical method of characteristics to achieve higher efficiency.


For any given $x(0)$ and $y(0)$, the characteristic curve for Eq. (\ref{eq:fp}) starting from this point is given by the
following equations
\begin{equation}\label{eq:ch}
      x'(t) = K(x,y), \quad 
      y'(t) = J(x,y).
\end{equation}
It can be derived from \eqref{eq:fp} and \eqref{eq:ch} that on these characteristic curves, the probability density function $P$ satisfies the following differential equation:
\begin{equation}\label{eq:fp3}
  \frac{d}{dt}\left(e^{\int_0^t f(x(s), y(s)) ds} P(x(t),y(t);t)\right) =
  e^{\int_0^t f(x(s), y(s)) ds}\partial_{xx}P(x(t), y(t); t),
\end{equation}
where $f(x, y) = \partial_x K(x,y) + \partial_y J(x,y)$. Based on such a form, we apply the backward Euler method to obtain the following semi-discretization of \eqref{eq:fp3}:
\begin{equation}\label{eq:fp-be}
    P(x(t_{k+1}),y(t_{k+1});t_{k+1}) - \Delta t \, \partial_{xx}P(x(t_{k+1}),y(t_{k+1});t_{k+1}) =
     e^{-\int_{t_k}^{t_{k+1}} f(x(s), y(s)) ds}P(x(t_k),y(t_k);t_k),
\end{equation}
where $\Delta t = t_{k+1} - t_k$ denotes the time step, and $(x(\cdot), y(\cdot))$ denotes a single characteristic curve. Therefore if we want to obtain $P(x_l, x_j; t_{k+1})$ for some specific point $(x_l, x_j)$, we need to solve the characteristic curve between $t_k$ and $t_{k+1}$ by finding the solution of the following backward system of ordinary differential equations:
\begin{equation}\label{eq:ch1}
  \left\{
    \begin{array}{@{}l}
      x'(t) = K(x,y), ~~ x(t_{k+1}) = x_l, \\[6pt]
      y'(t) = J(x,y), ~~ y(t_{k+1}) = y_j,
    \end{array}
  \right.
\end{equation}
so that in the last term of \eqref{eq:fp-be}, the values of $x(t_k)$ and $y(t_k)$ can be determined, and the integral of $f$ can be computed. In our scheme, we solve the backward ODE system \eqref{eq:ch1} by the classic Runge-Kutta scheme, and integrate $f$ using Simpson's rule. In fact, since $K$ and $J$ are independent of $t$, for any given $(x_l, x_j)$, the point $(x(t_k), y(t_k))$ and the integral of $f$ does not change if $\Delta t$ does not change. Therefore, in our implementation, we just choose a fixed time step so that these quantities need to be computed only once for each spatial grid point $(x_l, x_j)$. Numerically, the integral of $P(x,y;t)$ may deviate from one, and we scale the whole function after each time step by multiplying a constant to restore this property.

For the spatial discretization in our simulations, we adopt the finite difference method and the Fourier spectral method in different cases. The Fourier spectral method provides good accuracy for the derivatives, which are needed in the asymptotic expansions to be studied in \cref{sec:asymp}; while in the study of the decay of the probability, we adopt the finite difference method to avoid aliasing error appearing when periodizing the domain in the Fourier spectral method, which may ruin the tail of the probability density function. In what follows, we provide some details of these two methods.

\subsubsection{Finite difference scheme}
We adopt the uniform grid with $N \times M$ cells, each of which has the size of $\Delta x$ and $\Delta y$ in the $x$ and $y$ directions, respectively. Suppose that at time $t_k$ the probability is $P(x,y;t_k)$, and we want to determine $P$ at time $t_{k+1}$ with a central difference scheme to approximate $\partial_{xx}$. Then the full discretization of \eqref{eq:fp3} at point $(x_l,y_j)$ is given as 
\begin{equation}\label{eq:fp-fd}
    P(x_l,y_j;t_{k+1}) - \frac{\Delta t}{(\Delta x)^2} \left[ 
    P(x_{l+1},y_j;t_{k+1}) - 2P(x_l,y_j;t_{k+1}) + P(x_{l-1},y_j;t_{k+1})\right]
    = \lambda_{l,j}(\Delta t) P(\tilde{x}_l,\tilde{y}_j;t_k),
\end{equation}
where $(\tilde{x}_l, \tilde{y}_j) = (x(t_k), y(t_k))$ is obtained by solving the equation \eqref{eq:ch1}, and $\lambda_{l,j}(\Delta t)$ is the exponential of the integral of $f$ in \eqref{eq:fp-be}. For any points locating outside the computational domain, we set the value of $P$ to be zero. In general, the point $(\tilde{x}_l, \tilde{y}_j)$ is not on the grid point, and the value of $P(\tilde{x}_l, \tilde y_j;t_k)$ is obtained from the bilinear interpolation of $P(x,y;t_k)$. Defining
\begin{equation*}
  \mathbf{P}^{(k+1)}_j = \left[P(x_1,y_j;t_{k+1}), P(x_2,y_j;t_{k+1}),
    \cdots, P(x_N,y_j;t_{k+1})  \right]^{\top}, \quad j = 1,\cdots,M,
\end{equation*}
by \eqref{eq:fp-fd}, we are required to solve the linear systems $S \mathbf{P}^{(k+1)}_j = \mathbf{b}^{(k+1)}_j, \quad j = 1,\cdots,M$ with $S \in \mathbb{R}^{N\times N}$ being a tri-diagonal matrix and $\mathbf{b}^{(k+1)}_j$ corresponding to the right-hand side of \eqref{eq:fp-fd}. These tri-diagonal linear systems can be efficiently solved by the Thomas algorithm.

\subsubsection{Fourier spectral method} \label{sec:Fourier}
To employ the Fourier spectral method, the probability is assumed to be periodic in both real and imaginary directions. This is reasonable if we choose a sufficiently large computational domain such that the probability is sufficiently small on the boundary. Suppose the domain be $[-L_x/2,L_x/2]\times [-L_y/2, L_y/2]$, and the number of Fourier coefficients be $N,M$ in $x,y$ directions, respectively. The probability density function is approximated by
\begin{equation} \label{eq:Fourier}
    P(x,y;t) = \frac{1}{NM} \sum_{n=-N/2}^{N/2-1}
    \sum_{m=-M/2}^{M/2} \hat{P}_{n,m}(t) e^{\frac{2 \pi i}{L_x} n x}
    e^{\frac{2 \pi i}{L_y} m y}.
\end{equation}
Thus we can write down the equation \eqref{eq:fp-be} for $x(t_{k+1}) = x_l$ and $y(t_{k+1}) = y_j$ as
\begin{equation*}
    \frac{1}{NM} \sum_{n=-N/2}^{N/2-1}
    \sum_{m=-M/2}^{M/2} \left( 1+\frac{4\pi^2 n^2}{L_x^2} \Delta t\right) \hat{P}_{n,m}(t_{k+1}) e^{\frac{2 \pi i}{L_x} n x_l}
    e^{\frac{2 \pi i}{L_y} m y_j} = \lambda_{l,j}(\Delta t) P(\tilde{x}_l, \tilde{y}_j; t_k).
\end{equation*}
Applying discrete inverse Fourier transform on both sides, we get the scheme
\begin{equation}\label{eq:fp-fourier}
    \hat P_{n,m}(t_{k+1}) = \left( 1+\frac{4\pi^2 n^2}{L_x^2} \Delta t\right)^{-1} \sum_{l=1}^{N}\sum_{j=1}^M \lambda_{l,j}(\Delta t) P(\tilde{x}_l,\tilde{y}_j; t_k) e^{-\frac{2 \pi i}{L_x} n x_l}
    e^{-\frac{2 \pi i}{L_y} m y_j}.
\end{equation}
Note that the computational cost of the above scheme is $O(M^2N^2)$ since $\tilde{x}_l$ and $\tilde{y}_j$ are not collocation points.



\subsection{Decay of the distribution} \label{sec:decay}
To resolve the ambiguity about the decay rate of steady-state probability density function $P(x,y) := P(x,y;\infty)$, we solve the FP equation for a sufficiently long time until the steady state is attained. For the purpose of visualization, we consider the marginal probability density functions
\begin{equation}\label{eq:toy1-pxy}
    P_x(x) = \int_{-\infty}^{\infty} P(x,y) \,dy, 
    \quad P_y(y) = \int_{-\infty}^{\infty} P(x,y) \,dx,
\end{equation}
and we focus mainly on the cases with $B$ close to $\sqrt{3}$ where conflicting results are observed in the literature as stated in \cref{sec:intro-failcl}. The quantities \eqref{eq:toy1-pxy} are plotted in \cref{fig:toy1-Pxy} for $B$ from $1.5$ to $2.3$. From the figure, we observe the following phenomena: i) $P_y(y)$ drops rapidly for $B\leq 1.7$; ii) when $B$ surpasses $1.8$, the tail of $P_y(y)$ starts to rise up; iii) for $B\ge 2.1$, both $P_y(y)$ and $P_x(x)$ show the power-like decay. These results contradict the statement in \cite{nagata2016justification} that the power decay shows up only when $B$ is greater than $2.6$. In our results, such decay is obvious as early as $B = 2.0$. In fact, we conjecture that such power decay appears immediately when $B$ exceeds $\sqrt{3}$. It is not obvious in the numerical experiments only because of the small coefficient in front of the power decay. Such an argument can be supported by some analysis of the FP equation, which will be clarified in the following two parts.

\begin{figure}[!h]
  \centering
  \includegraphics[width=0.35\textwidth]{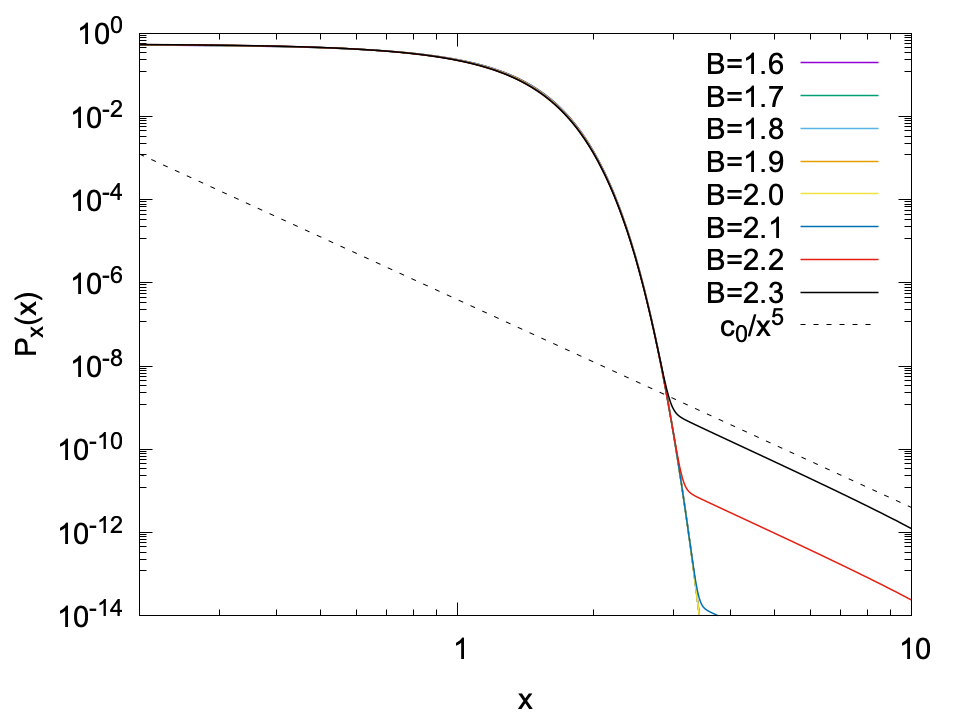}\qquad
  \includegraphics[width=0.35\textwidth]{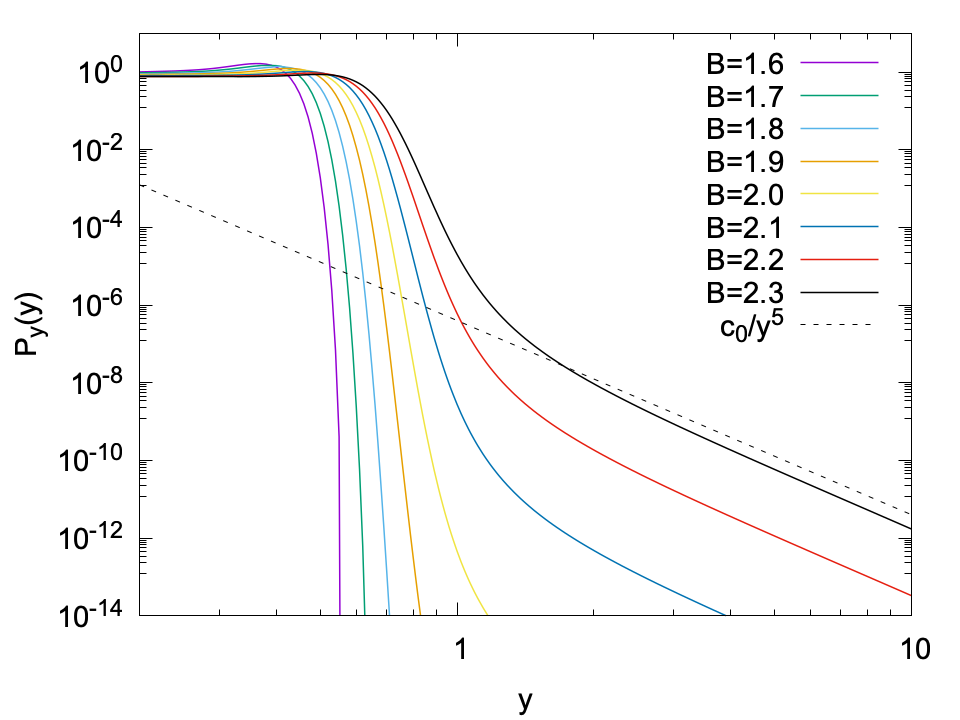}
  \caption{$P_x(x)$ and $P_y(y)$ for different $B$.\label{fig:toy1-Pxy} }
\end{figure}

\subsubsection{Possibility of the exponential decay ($B\leq \sqrt{3}$)}
{In the steady-state Fokker-Planck equation $\partial_x (KP) + \partial_y (JP) = \partial_{xx} P$, both $K$ and $J$ are polynomials, as inspires us to conjecture that the solution $P(x,y)$ may behave like the exponential of a polynomial when $x$ or $y$ is large, so that $P(x,y)$ decays exponentially.} Specifically, we can write $P(x,y)$ as
\begin{displaymath}
P(x,y) \approx \exp(-A r^{\beta}) g(\theta) \qquad \text{for } r \gg 1,
\end{displaymath}
where $(r,\theta)$ is the polar coordinates of $(x,y)$. Straightforward computation yields
\begin{align*}
L^T P &\approx \bigg( \frac{g(\theta )}{2} \left[\beta ^2
   A r^{\beta-2} (Ar^{\beta}-1) + \left(\beta ^2 A^2
   r^{2 \beta-2}-2 \beta  A
   r^{\beta +2}-(\beta -2)
   \beta  A r^{\beta-2}+12
   r^2\right)\cos 2\theta - 2 \beta  A
   r^{\beta}+4\right] \\
& \quad +g'(\theta ) \left[ \left(\beta  A
   r^{\beta-2}+r^2+r^{-2}\right)+B
   \right] \sin 2\theta +g''(\theta) r^{-2} \sin^2 \theta \bigg) \exp(-A r^{\beta}), \qquad \text{for } r \gg 1.
\end{align*}
When $r \rightarrow +\infty$, the term with slowest decay behaves like
\begin{equation*}
\varphi(r) = r^{\max(\beta+2,2\beta-2)} \exp(-A r^{\beta}).
\end{equation*}
By focusing on this leading order term, we have
\begin{equation*}
\frac{L^T P(x,y)}{\varphi(r)} \approx \left\{ \begin{array}{@{}ll}
  -\beta A g(\theta) \cos 2\theta, & \text{if } \beta < 4, \\[3pt]
  -\beta A g(\theta) \cos 2\theta + \beta^2 A^2 g(\theta) \cos^2 \theta, & \text{if } \beta = 4, \\[3pt]
  \beta^2 A^2 g(\theta) \cos^2 \theta, & \text{if } \beta > 4.
\end{array} \right.
\end{equation*}
Since $P(x,y)$ is the steady-state solution of the Fokker-Planck equation, the above quantity must equal zero. For any $\alpha$, this requires that $g(\theta)$ be zero for almost every $\theta$.

If $\beta < 4$, the value of $g(\theta)$ can be nonzero for $\theta = \pm \pi/4$ and $\theta = \pm 3\pi/4$. This means that $P(x,y)$ can be nonzero in two strips parallel to the lines $y = \pm x$. However, such strips cannot be formed due to the diffusion in the $x$ direction. Similarly, if $\beta > 4$, $P(x,y)$ can have nonzero values in the strip perpendicular to the $x$-axis, which is also not allowed because of the Brownian motion. This excludes the choices $\beta < 4$ and $\beta > 4$.

In fact, such an FP equation only allows the localizing strip to be parallel to the $x$-axis, which corresponds to $\theta = 0$ and $\theta = \pi$. When $\beta = 4$, choosing $A = 1/4$ allows us to have $g(\theta) \neq 0$ when either $\theta = 0$ or $\theta = \pi$ holds. As a summary, such analysis shows that if $P(x,y)$ has exponential decay, the only possible choice of $\beta$ is $4$, and in this case, the support of $P(x,y)$ must be confined in a strip-like domain parallel to the $x$-axis. This occurs when $B \leq \sqrt{3}$, as can be demonstrated in the following theorem:

\begin{proposition}\label{thm:toy1support}
  Suppose $0\leq B\leq \sqrt{3}$. There exists a constant $\alpha > 0$ such that $J(x,y)$ defined in \eqref{eq:toy1Ky} satisfies the following conditions:
  \begin{itemize}
      \item[i)] $J(x,\alpha) \leq 0$ for all $x \in \mathbb{R}$.
      \item[ii)] $J(x,-\alpha) \geq 0$ for all $x \in \mathbb{R}$.
  \end{itemize}
\end{proposition}
\begin{proof}
  Here we only show that $J(x,\alpha) \leq 0$ for all $x \in \mathbb{R}$, and the proof of the other part is almost identical. For simplicity, we define
  \begin{equation*}
      Q(x) := J(x,\alpha) = -3 \alpha x^2 - Bx + \alpha^3 -\alpha.
  \end{equation*}
  The function $Q(x)$ is a quadratic polynomial, whose maximum value can be obtained as
  \begin{equation}\label{eq:Jmax}
      \max_{x\in \mathbb{R}} Q(x) = \frac{B^2}{12\alpha} + \alpha^3 -\alpha
      = \frac{1}{\alpha}\left[ \left(\alpha^2-\frac{1}{2}\right)^2 + \frac{B^2}{12}-\frac{1}{4} 
      \right].
  \end{equation}
  When $0\leq B \le \sqrt{3}$, we can choose
  \begin{equation}\label{eq:supp}
      \alpha = \frac{1}{\sqrt{2}}\sqrt{1-\sqrt{1-\frac{B^2}{3}}},
  \end{equation}
  so that $\max_{x\in \mathbb{R}} Q(x)$ is exactly zero, meaning that $Q(x)$ is always non-positive, which completes the proof of the proposition.
\end{proof}

This proposition shows that when $0\le B \le \sqrt{3}$, the solution of \eqref{eq:sto} satisfies $y(t) \in [-\alpha,\alpha]$ if the initial condition $y(0) \in [-\alpha,\alpha]$. This can be illustrated by plotting the velocity field $(K(x,y), J(x,y))$, which shows that on the lines $y = \pm \alpha$, all the velocities point toward the horizontal axis. Thus $(x(t), y(t))$ can never drift out of the strip between these two lines, causing zero values of $P(x,y)$ for all $|y| > \alpha$. In other words, the distribution $P(x,y)$ has a compact support $[-\alpha,\alpha]$ in the imaginary direction, and our previous analysis shows that the decay rate in the $x$-direction is like $\exp(-x^4/4)$. The right panel of \cref{fig:toy1-Pxy} also validates the existence of such a strip.
 
\subsubsection{Possibility of the power decay ($B > \sqrt{3}$)}
For completeness, we apply the similar analysis to demonstrate the rate of the power decay. Such analysis has already been done in \cite{aarts2013localised}, while the angular function is not included. Here we will carry out a more rigorous analysis by assuming
\begin{equation*}
P(x,y) \approx r^{-\beta} g(\theta) \qquad \text{for } r\gg 1.
\end{equation*}
Thus
\begin{align*}
L^T P &\approx r^{-\beta -2} \bigg(g''(\theta) \sin
   ^2(\theta) + g'(\theta)
   \left(B r^2+
   \left(\beta
   +r^4+1\right)\sin 2\theta \right) \\
 & \quad +\frac{1}{2} g(\theta)
   \left[\beta ^2+ \left(\beta  (\beta +2)-2
   (\beta -6) r^4\right)\cos (2 \theta)
   -2 (\beta -2) r^2\right]\bigg) \\
& \approx [(6-\beta) g(\theta) \cos 2\theta + g'(\theta) \sin 2\theta] r^{2-\beta}.
\end{align*}
Therefore $(6-\beta) g(\theta) \cos 2\theta + g'(\theta) \sin 2\theta = 0$, whose solution is $g(\theta) = C (\sin 2\theta)^{\beta/2-3}$ for any $C \in \mathbb{R}$. Only when $\beta = 6$, the positivity of $P(x,y)$ can be guaranteed. Thus we conclude that
\begin{equation*}
    P(x,y) \approx \frac{C}{(x^2+y^2)^3}, \qquad \text{for } x^2+y^2 \gg 1,
\end{equation*}
which is clearly not a localized probability density function, and should correspond to any $B$ exceeding the critical value $\sqrt{3}$. As a result, the decay of the marginal probability density functions \eqref{eq:toy1-pxy} behave like
\begin{gather*}
    P_x(x) \propto x^{-5}, \qquad \text{for } |x| \gg 1, \\
    P_y(y) \propto y^{-5}, \qquad \text{for } |y| \gg 1,
\end{gather*}
which has been numerically verified as shown in \cref{fig:toy1-Pxy}. The analysis further confirms that the absence of power decay in the numerical results of $B=1.8$ and $1.9$ is due to the smallness of $C$. As we will see later in Section \ref{sec:disappearance}, the values of the probability density function may have dropped below the machine epsilon when the power decay shows up, so that the numerical method is not able to capture such a decay rate. 

\subsection{Effect of the fat-tailed distribution} \label{sec:effect}
Knowing that $B = \sqrt{3}$ separates the two types of decay rates, we would like to study how this affects the observables. When $B$ is large, since the CL method no longer converges to the exact integral, we conclude that at least one of the assumptions \ref{itm:H1}--\ref{itm:H3} is violated. Among the three conditions, the only one that may related to the tail of $P(x,y)$ is \ref{itm:H1}. In fact, in \cite{aarts2010complex}, instead of given as a condition, the assumption \ref{itm:H1} is derived as follows:
\begin{equation} \label{eq:dtau_OP}
\begin{split}
\frac{\partial}{\partial \tau} \int_{\mathbb{R}} \int_{\mathbb{R}}
  \mathcal{O}(x,y;\tau) P(x,y;t-\tau) \,dx \,dy &=
\int_{\mathbb{R}} \int_{\mathbb{R}} \left[
  \frac{\partial \mathcal{O}(x,y;\tau)}{\partial \tau} P(x,y;t-\tau) +
  \mathcal{O}(x,y;\tau) \frac{\partial P(x,y;t-\tau)}{\partial \tau}
\right] \,dx \,dy \\
&= \int_{\mathbb{R}} \int_{\mathbb{R}}
  [P(x,y;t-\tau) L\mathcal{O}(x,y;\tau) - \mathcal{O}(x,y;\tau) L^T P(x,y;t-\tau)] \,dx \,dy,
\end{split}
\end{equation}
which equals zero due to the formal mutual adjointness of $L$ and $L^T$. However, it has also been pointed out in \cite{aarts2010complex, scherzer2019complex} that the above quantity may not vanish if $P(x,y;t)$ does not have sufficient decay. Specifically, in order that \eqref{eq:dtau_OP} equals zero, we need the following limits:
\begin{align}
\label{eq:limits1}
& \lim_{X\rightarrow \infty} \int_{\mathbb{R}} \frac{\partial \mathcal{O}(X,y;\tau)}{\partial x} P(X,y;t-\tau) \,dy = \lim_{X\rightarrow \infty} \int_{\mathbb{R}} \frac{\partial P(X,y;\tau)}{\partial x} \mathcal{O}(X,y;t-\tau) \,dy = 0, \\
\label{eq:limits2}
 & \lim_{X\rightarrow \infty} \int_{\mathbb{R}} K(X,y) \mathcal{O}(X,y;\tau) P(X,y;t-\tau) \,dy = \lim_{Y\rightarrow \infty} \int_{\mathbb{R}} J(x,Y) \mathcal{O}(x,Y;t-\tau) P(x,Y;\tau) \,dx = 0.
\end{align}
Only when these limits hold for all $t$ and $\tau$, we can ensure that integration by parts without boundary terms can be carried out to show that \eqref{eq:dtau_OP} equals zero.

We focus on the second integral in \eqref{eq:limits2} and the other limits can be considered in a similar way. Since the limit \eqref{eq:limits2} must hold for all $t$ and $\tau$, a necessary condition for the validity of the CL method can be obtained by setting $t=\tau$ and letting $\tau$ tend to infinity, which yields 
\begin{equation}\label{eq:Ey}
    \lim_{y\rightarrow \infty} E(y)  = 0, \quad E(y) := \int_{\mathbb{R}} J(x,y) \mathcal{O}(x,y;0) P(x,y;\infty) \, dx. 
\end{equation}
Here $P(x,y;\infty)$ is just the function $P(x,y)$ as stated in the beginning of Section \ref{sec:decay}, and now it is clear that \eqref{eq:Ey} holds only when $P(x,y)$ has sufficient decay when $y \to \infty$. In our case, when $B > \sqrt{3}$ and $|y|$ is large,
\begin{equation*}
E(y) \approx -\int_{\mathbb{R}} (y+B x+3x^2 y - y^3) (x^2-y^2+i2xy)
  \frac{C}{(x^2+y^2)^3} \,dx = -\frac{C \pi (i B+4y^2-1)}{4 y^2},
\end{equation*}
whose real part does not vanish as $y$ tends to infinity. Similarly, the other limit in \eqref{eq:limits2} does not hold either. As a result, biased results are generated. Such phenomenon has also been numerically validated in \cref{fig:toy1-kypo}. When $B \leq \sqrt{3}$, due to the fast decay of $P(x,y)$, we observe unbiased results in the simulations.

\begin{figure}[!ht]
    \centering
    \includegraphics[width=0.35\textwidth]{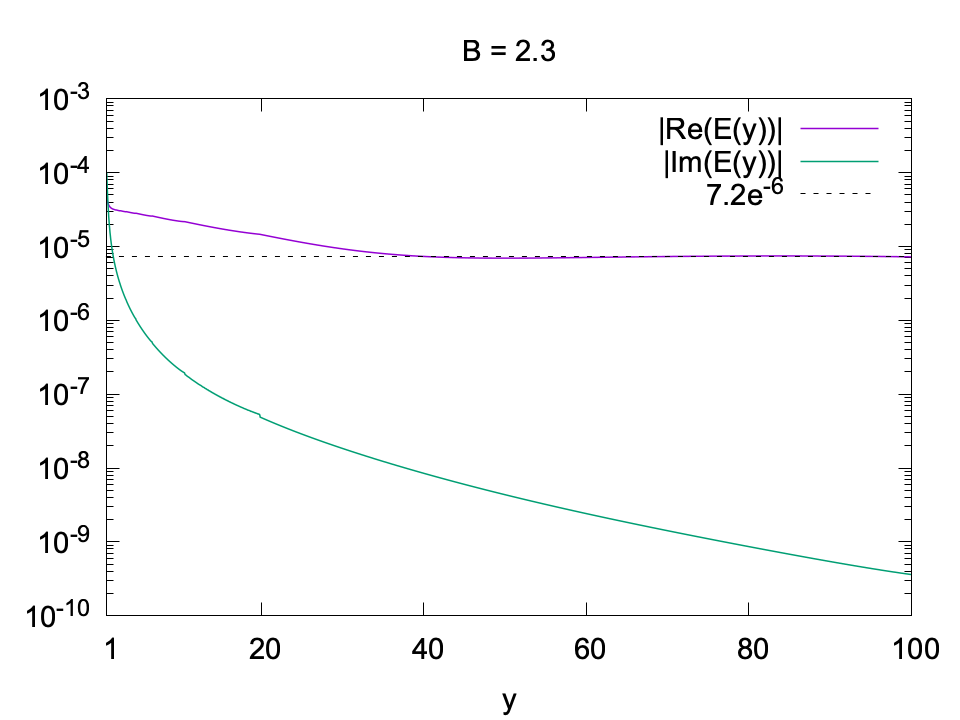}
    \caption{The decay of $E(y)$ for $B=2.3$.}
    \label{fig:toy1-kypo}
\end{figure}

{%
\begin{remark} \label{rem:holomorphic}
If the action $S$ is only meromorphic, meaning that the velocities $K$ and $J$ may contain poles, then the conditions \eqref{eq:limits1}\eqref{eq:limits2} must be supplemented by the corresponding boundary conditions at poles. We refer the readers to a recent paper \cite{seiler2020complex} for some discussions on such cases.
\end{remark}}



\subsection{Asymptotics near $B=\sqrt{3}$} \label{sec:asymp}
By this example, we would like to reveal more on what happens at the critical point $B = \sqrt{3}$. \cref{fig:toy1cl} shows that the two curves, which initially coincide, eventually separate from each other as $B$ increases, meaning that at least one of the curves is not analytic at point $B = \sqrt{3}$. In this section, we would like to study the asymptotics around this point and explain how the analyticity fails.

The most straightforward idea to study the asymptotics is to set $B^{\epsilon} = \sqrt{3} - \epsilon$ and expand the associated probability density function $P^{\epsilon}(x,y)$ by
\begin{equation}\label{eq:expansion1}
  P^{\epsilon}(x,y) = P_0(x,y) + \epsilon P_1(x,y) + \epsilon^2 P_2(x,y) + \cdots.
\end{equation}
By setting $\epsilon=0$, we see that $P^0(x,y)$ corresponds to the probability density function for $B = \sqrt{3}$, and it has been shown in the proof of Proposition \ref{thm:toy1support} that
\begin{equation}\label{eq:support}
  \supp_y P^{\epsilon} = [-\alpha^{\epsilon}/\sqrt{2},
  \alpha^{\epsilon}/\sqrt{2}],\quad \alpha^{\epsilon} 
  = \sqrt{1-\sqrt{1-(B^{\epsilon})^2/3}}.
\end{equation}
Unfortunately, such an expansion does not converge due to the following proposition:

\begin{proposition}\label{thm:support}
  Let $P^{\epsilon}(y)$ be a class of functions defined for every $\epsilon \in (0, \delta)$, and the functions satisfy $\supp P^{\epsilon}(y) = [-\alpha^{\epsilon},\alpha^{\epsilon}]$. Then there does not exist a sequence of functions $\{P_0(y), P_1(y), \cdots\}$ such that the infinite series
  \begin{equation}\label{eq:f-expansion1}
    \sum_{n=0}^{+\infty}\epsilon^n P_n(y)
  \end{equation}
  converges pointwisely to $P^{\epsilon}(y)$ for any $\epsilon \in (0, \delta)$.
\end{proposition}
\begin{proof}
  Note that $\alpha^{\epsilon}$ monotonically decreases as $\epsilon$ increases. Suppose the series \eqref{eq:f-expansion1} converges to $P^{\epsilon}(y)$ for any $\epsilon \in (0,\delta)$, we know that for
  $\epsilon \in (\delta/2,\delta)$ and
  $y \in (\alpha^{\delta/2}, \alpha^0)$, it holds that
  \begin{equation*}
    P^{\epsilon}(y) = \sum_{n=0}^{+\infty}\epsilon^n P_n(y) = 0.
  \end{equation*}
  Regarding the series as the power series with respect to $\epsilon$, we know that for $y \in (\alpha^{\delta/2}, \alpha^0)$, the value of $P^{\epsilon}(y)$ must be zero for any $\epsilon\in (0,\delta)$. This contradicts the
  assumption that $\supp P^{\epsilon} = [-\alpha^{\epsilon},\alpha^{\epsilon}]$ for $\epsilon \in (0,\delta/2)$.
\end{proof}

The above result shows that in \cref{fig:toy1cl}, the curve of observable predicted by the CL method is likely to be non-analytic. Below we will provide a legitimate asymptotic expansion for $P^{\epsilon}(x,y)$ based on the knowledge of its support.

\subsubsection{The asymptotic expansion}
To avoid the divergence problem arising from the variation of the support, we are going to scale the variable $y$ to align the support of $P^{\epsilon}(x,y)$ for any $\epsilon$. Precisely, we let $a^{\epsilon} = 1/\alpha^{\epsilon}$ and construct a new function $\tilde{P}^{\epsilon}$ as
\begin{equation}\label{eq:scale}
  \tilde{P}^{\epsilon}(x,y) = P^{\epsilon}(x,a^{\epsilon}y),
\end{equation}
such that $\supp_y \tilde P^{\epsilon} = [-1/\sqrt{2},1/\sqrt{2}]$ for any $\epsilon$, and 
the original probability density function can be reconstructed as $P^\epsilon(x,y) = \tilde P^\epsilon(x,\alpha^{\epsilon}y)$. Thus the governing equation of $P^{\epsilon}(x,y)$ is
\begin{equation}\label{eq:FPEtildeP}
  \frac{\partial}{\partial t} \tilde P^{\epsilon}  +
    \frac{\partial}{\partial x}(\tilde K^{\epsilon}\tilde P^{\epsilon}) + 
    \alpha^{\epsilon} \frac{\partial}{\partial y}(\tilde J^{\epsilon}\tilde P^{\epsilon})
    = \frac{\partial^2}{\partial x^2}\tilde{P}_{\epsilon}, 
\end{equation}
in which $\tilde K^\epsilon(x,y) := K(x,a^{\epsilon}y)$ and $\tilde J^\epsilon(x,y) :=J(x,a^{\epsilon}y)$. Note that the Taylor expansion of $a^{\epsilon}$ is
\begin{equation} \label{eq:aexpansion2}
a^{\epsilon} = 1 + a_1 \epsilon^{\frac{1}{2}} + a_2 \epsilon + \cdots, \qquad
a_1 = \frac{\sqrt{2}}{\sqrt[4]{3}}, \quad a_2 = \frac{2}{\sqrt{3}}, \quad \dots.
\end{equation}
{It is worth noting that this expansion is legitimate only for positive $\epsilon$, and} we should therefore expand all the quantities with respect to $\sqrt{\epsilon}$ instead of $\epsilon$. Let
\begin{equation} \label{eq:expansion2}
\tilde P^{\epsilon} = \sum_{k=0}^{+\infty} \epsilon^{\frac{k}{2}}\tilde P_{\frac{k}{2}}, \quad
K^{\epsilon} = \sum_{k=0}^{+\infty} \epsilon^{\frac{k}{2}} \tilde K_{\frac{k}{2}}, \quad
J^{\epsilon} = \sum_{k=0}^{+\infty} \epsilon^{\frac{k}{2}} \tilde J_{\frac{k}{2}}.
\end{equation}
One can figure out all the terms $\tilde{K}_{\frac{k}{2}}$ and $\tilde{J}_{\frac{k}{2}}$ by the analytical expressions of $K(x,y)$ and $J(x,y)$. Then by balancing the terms with various orders of $\epsilon$ in the equation \eqref{eq:FPEtildeP}, we can obtain the equations for $\tilde{P}_{\frac{k}{2}}$:
\begin{equation} \label{eq:sol}
\begin{aligned}
\mathcal{O}(1): \quad &
    \frac{\partial}{\partial t} \tilde P_0 
    + \frac{\partial}{\partial x}(\tilde K_0 \tilde P_0) 
    + \frac{\partial}{\partial y}(\tilde J_0 \tilde P_0) 
    = \frac{\partial^2}{\partial x^2}{\tilde P_0}; \\
\mathcal{O}(\epsilon^{1/2}): \quad &
    \frac{\partial}{\partial t} \tilde P_{\frac{1}{2}} 
    + \frac{\partial}{\partial x}(\tilde K_{0} \tilde P_{\frac{1}{2}}) 
    + \frac{\partial}{\partial y}(\tilde J_{0} \tilde P_{\frac{1}{2}}) 
    + \frac{\partial}{\partial x}(\tilde K_{\frac{1}{2}} \tilde P_0) 
    + \frac{\partial}{\partial y}
      ( \tilde J_{\frac{1}{2}} \tilde P_0 - a_{\frac{1}{2}} \tilde J_0 \tilde P_0) 
    = \frac{\partial^2}{\partial x^2}\tilde{P}_{\frac{1}{2}}; \\
\mathcal{O}(\epsilon): \quad &
    \frac{\partial}{\partial t} \tilde P_{1}  
    + \frac{\partial}{\partial x}(\tilde K_{0} \tilde P_1) 
    + \frac{\partial}{\partial y}(\tilde J_{0} \tilde P_1)  
    + \frac{\partial}{\partial x}
      ( \tilde K_{\frac{1}{2}} \tilde P_{\frac{1}{2}} + \tilde K_1 \tilde P_0 ) \\
    & \quad + \frac{\partial}{\partial y} 
      \left( \tilde J_{1}\tilde P_0  + \tilde J_{\frac{1}{2}} \tilde P_{\frac{1}{2}}  
       - a_{\frac{1}{2}} \tilde J_0 \tilde P_{\frac{1}{2}} 
       + (a_{\frac{1}{2}}^2 - a_1) \tilde J_0 \tilde P_0 \right) 
    = \frac{\partial^2}{\partial x^2} \tilde{P}_{1}; \\
\cdots \qquad & \cdots \qquad \cdots \qquad \cdots \qquad \cdots \qquad \cdots \qquad \cdots \qquad \cdots \qquad \cdots \qquad \cdots
\end{aligned}
\end{equation}
Thus we can obtain each $\tilde{P}_{k/2}$ by solving these equations. Due to the appearance of $\epsilon^{1/2}$, such expansion cannot be extended to negative $\epsilon$, as also implies the non-analyticity of the solution provided by the CL method.

The convergence of the series \eqref{eq:expansion2} can be verified numerically. Instead of solving \eqref{eq:sol} directly, we adopt an alternative way to find the functions $\tilde{P}_{k/2}$. In fact, these quantities are related to the formal expansion \eqref{eq:expansion1}, whose terms satisfy the equations
\begin{equation}\label{eq:fp-expansion1}
    \frac{\partial}{\partial t} P_l 
    + \frac{\partial}{\partial x}(K_0 P_l) + \frac{\partial}{\partial y}(J_0 P_l) 
    - \frac{\partial}{\partial x}(y P_{l-1}) + \frac{\partial}{\partial y}(x P_{l-1}) 
    = \frac{\partial^2}{\partial x^2}{P_l}, \quad l = 1,2,\dots,
\end{equation}
which can be derived by inserting \eqref{eq:expansion1} into the FP equation \eqref{eq:fp} and balancing the terms with the same orders of $\epsilon$. It can be verified that by setting
\begin{equation} \label{eq:P_to_tildeP}
    \tilde P_0 = P_0,\qquad 
    \tilde P_{\frac{1}{2}} = a_{\frac{1}{2}}y \frac{\partial P_0}{\partial y}, \qquad 
    \tilde P_{1} = P_1 + a_1 y \frac{\partial P_0}{\partial y} 
    + \frac{1}{2} a_{\frac{1}{2}}^2 y^2 \frac{\partial^2 P_0}{\partial y^2},\qquad \cdots,
\end{equation}
we can obtain the solutions to the equations \eqref{eq:sol}. The method we adopt is to solve \eqref{eq:fp-expansion1} numerically, and then use \eqref{eq:P_to_tildeP} to convert the results to $\tilde{P}_{k/2}$. In this process, high-order derivatives of the solutions are needed in \eqref{eq:P_to_tildeP}, which requires us to adopt the Fourier spectral method described in \cref{sec:Fourier} to find the numerical solutions.


The computational details are stated as follows. The computational domain is set to be $[-5,5]\times [-1,1]$, which is sufficiently large since $P^{\epsilon}(x,y)$ decays fast in the $x$-direction and has a compact support in the $y$-direction. We choose $N = 480$ and $M = 240$ in \eqref{eq:Fourier}, and after solving $P_0$ from the original FP equation, we compute $P_1$ to $P_5$ successively by solving \eqref{eq:fp-expansion1}. All the equations are solved until a steady state is attained. Based on these functions, one can compute $\tilde{P}_{{k}/{2}}$ up to $k = 11$. Then the results are inserted back to the equation \eqref{eq:expansion2} with the infinite series truncated. In \cref{fig:toy1-expansion}, we plot the function $P^{\epsilon}(x,y)$ with $\epsilon = 0.02$ approximated by different truncations, which shows the convergence of the series of $\tilde{P}^{\epsilon}$ given in \eqref{eq:expansion2}. In particular, we observe that some oscillations appearing in the early truncations are suppressed as we increase the number of terms.

\begin{figure}[!ht]
    \centering
    \includegraphics[width=\textwidth]{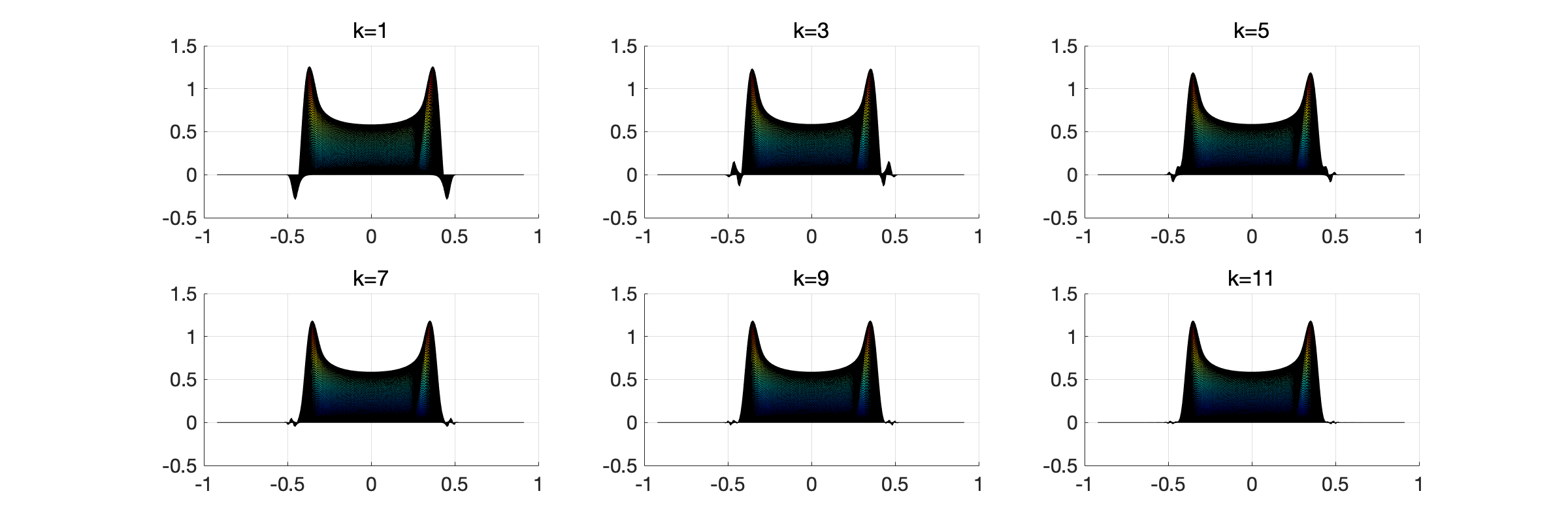}
    \caption{The distribution constructed by expansion \eqref{eq:expansion2}. \label{fig:toy1-expansion} }
\end{figure}

Now we study the analyticity of the curves shown in \cref{fig:toy1cl}. The expectation value for the observable can be related to the scaled function $\tilde{P}^{\epsilon}(x,y)$ by
\begin{equation*}
       \langle O \rangle^{\epsilon} = \int_{\mathbb{R}} \int_{\mathbb{R}}
       O(x+iy) P^{\epsilon}(x,y) \,dx \,dy  
       = a^{\epsilon} \int_{\mathbb{R}} \int_{\mathbb{R}} O(x + ia^{\epsilon}y) 
        \tilde P^{\epsilon}(x,y) \,dx \,dy.
\end{equation*}
The expansion of $O(x+ia^{\epsilon}y)$ can be obtained by substituting the expansion of $a^\epsilon$ \eqref{eq:aexpansion2} to the observable function $O(x+iy)=(x+iy)^2$, which then yields an expansion of $\langle O \rangle^{\epsilon}$:
\begin{equation}\label{eq:oexpansion2}
  \langle O \rangle^{\epsilon} = \sum_{k=0}^{+\infty}\epsilon^{\frac{k}{2}}\langle O \rangle_{\frac{k}{2}},
\end{equation}
where the coefficients $\langle O\rangle_{{k}/{2}}$ can be evaluated by:
\begin{align*}
    \langle O \rangle_0 &= \int_{\mathbb{R}} \int_{\mathbb{R}} O_0 \tilde P_0 \,dx \,dy, \qquad
    \langle O \rangle_{\frac{1}{2}} =  \int_{\mathbb{R}} \int_{\mathbb{R}} 
        (O_0 \tilde P_{\frac{1}{2}} + O_{\frac{1}{2}}\tilde P_0 
        + a_\frac{1}{2} O_0 \tilde P_0)\,dx \,dy,\\[6pt]
    \langle O \rangle_1 &= \int_{\mathbb{R}} \int_{\mathbb{R}} 
        \left( O_0 \tilde P_1 + O_\frac{1}{2} \tilde P_{\frac{1}{2}} + O_1 \tilde P_0
            + a_\frac{1}{2} (O_0 \tilde P_{\frac{1}{2}} + O_\frac{1}{2}\tilde P_0)
            + a_1 O_0 P_0 \right) \,dx \,dy, \qquad \cdots.
\end{align*}
Some numerical values of $\langle O \rangle_{{k}/{2}}$ are tabulated in \cref{tab:Ocoeff}, from which one can see that $\langle O \rangle_{{k}/{2}}$ with half-integer orders are all very small. Actually, using the relations \eqref{eq:P_to_tildeP}, one can show that
\begin{displaymath}
\langle O \rangle_{\frac{k}{2}} = \left\{ \begin{array}{@{}ll}
\displaystyle \int_{\mathbb{R}} \int_{\mathbb{R}} O(x+iy) P_{\frac{k}{2}}(x,y) \,dx \,dy, & \text{if } k \text{ is even}, \\[10pt]
0, & \text{if } k \text{ is odd}.
\end{array} \right.
\end{displaymath}
This indicates that the asymptotic expansion for the expectation $\langle O \rangle^{\epsilon}$ contains only integer orders, as allows us to extend the expansion to negative $\epsilon$, which matches the exact values of the integral (see \cref{fig:Ovse}). Unfortunately, the solution of the FP equation cannot adopt the form of the series \eqref{eq:expansion2} for $\epsilon < 0$, causing failure of the CL computation. 

\begin{table}[!ht]
  \centering
  \caption{Coefficients of $\langle O \rangle_{{k}/{2}}$ in
    \eqref{eq:oexpansion2}. \label{tab:Ocoeff}}
  \begin{tabular}{ll@{\qquad}ll}
    \toprule
    $\langle O \rangle_0$  & $0.3579 - 0.2284i$
    &$\langle O \rangle_{\frac{1}{2}}$   & $-2.4 \times 10^{-14} + 2.3 \times 10^{-15} i$\\\midrule
    $\langle O \rangle_1$  & $0.1136 + 0.0864i$
    &$\langle O \rangle_{\frac{3}{2}}$  & $3.1 \times 10^{-10} + 2.1 \times 10^{-12} i$\\\midrule
    $\langle O \rangle_2$  & $-0.0165 + 0.0360i$
    &$\langle O \rangle_{\frac{5}{2}}$  & $-3.6\times 10^{-7} + 2.0 \times 10^{-9} i$\\\midrule
    $\langle O \rangle_3$  & $-0.0089 - 0.0030i$
    &$\langle O \rangle_{\frac{7}{2}}$  & $-1.9\times 10^{-4} + 4.7 \times 10^{-9} i$\\
    \bottomrule
  \end{tabular}
\end{table}

\begin{figure}[!ht]
  \centering
  \includegraphics[width=0.35\textwidth]{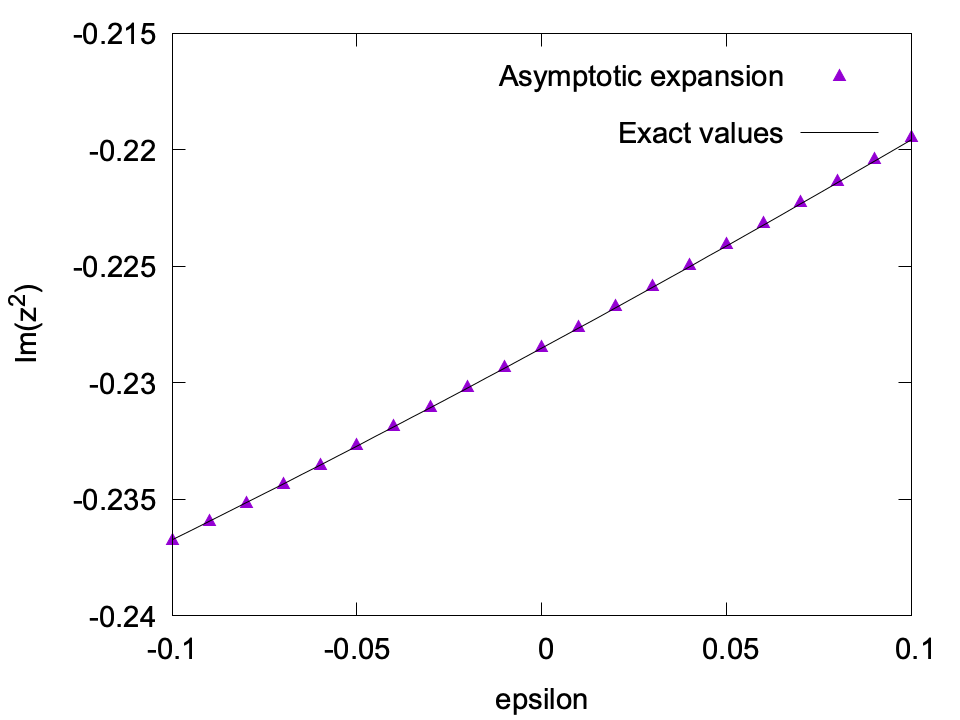}
  \caption{The expectations obtained from expansion \eqref{eq:expansion2} versus
    $\epsilon$.\label{fig:Ovse}}
\end{figure}

\subsubsection{Transition to the global support} \label{sec:disappearance}
Despite the non-analyticity of the observable predicted by the CL method, we expected that the transition is smooth when the global support appears. For $\epsilon < 0$, assume that
\begin{equation*}
P^{\epsilon}(x,y) \approx \frac{C^{\epsilon}}{(x^2+y^2)^3}, \qquad \text{for } x^2 + y^2 \gg 1.
\end{equation*}
We conjecture that the coefficient $C^{\epsilon}$ has the form $\exp(\alpha_0/\epsilon)$, with $\alpha_0$ being a constant. Such a form allows a $C^{\infty}$ transition from zero to a nonzero value, leading to the smoothness of the value of $\langle O \rangle^{\epsilon}$ predicted by the CL method. We verify this numerically by fitting the marginal probability density function $P_y(y)$ (defined in \eqref{eq:toy1-pxy}) for $\epsilon < 0$. It can be expected that $P_y(y) \propto C(y) \exp(\alpha_0 / \epsilon)$ for sufficiently large $y$. Here we pick $y = 1.5$ and $y = 2$, and do the curve fitting in \cref{fig:rexpansion}. Note that when $\epsilon$ is close to zero, the value of $P^{\epsilon}(x,y)$ is very small, so that the accuracy is affected by the round-off error. However, for $\epsilon < -0.4$, the numerical solution perfectly fits our conjecture. This indicates the $C^{\infty}$ transition from local support to global support.

\begin{figure}[htbp]
\centering
\includegraphics[width=0.35\textwidth]{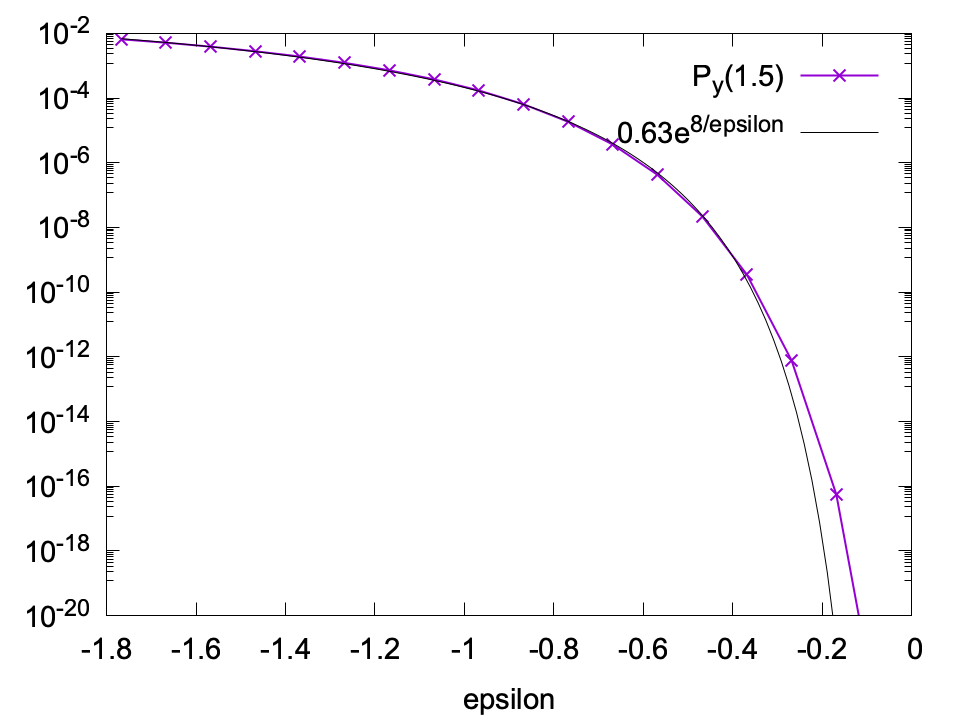}\qquad
\includegraphics[width=0.35\textwidth]{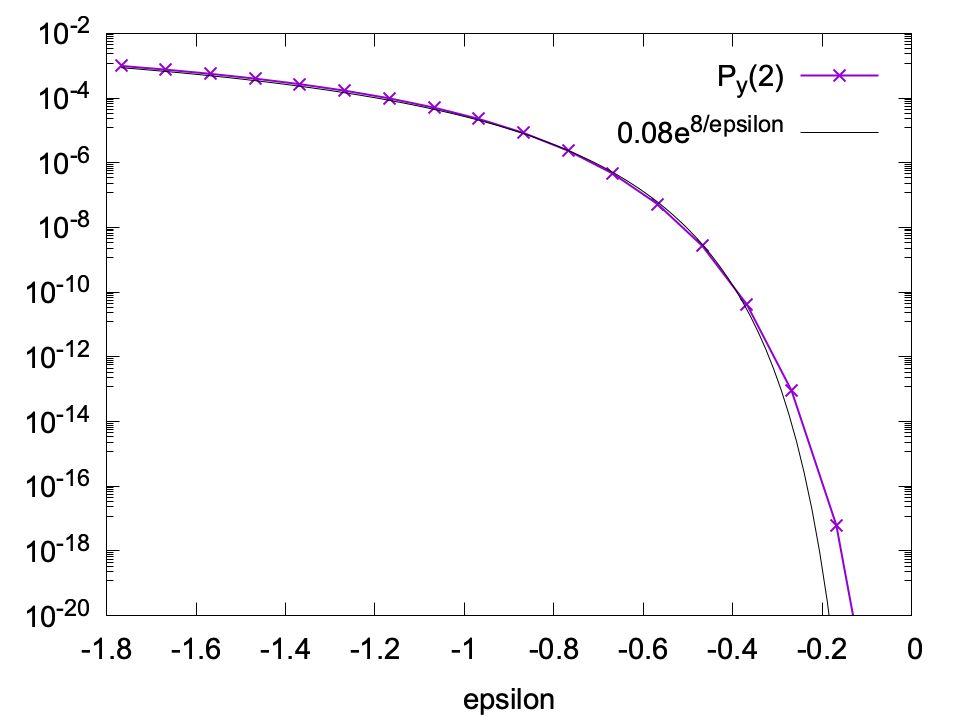}
\caption[]{$P_y$ at point $y=1.5$ and $2$ for different
  $\epsilon$. \label{fig:rexpansion}}
\end{figure}

\subsection{Implications of the model problem}
By a careful analysis of the model problem with action \eqref{eq:toy1action}, we have gained better understanding of some properties of the CL method. In general, the CL method looks quite fragile. For every CL result that looks convergent, we have to analyze the decay of the probability density function to show its validity. This can be done for some simple cases. For example, in the one-dimensional case with $S(x)$ dominated by the monomial $x^k$, we can use the same method as in Section \ref{sec:decay} to show that the decay rate is $P(x,y) \approx (x^2+y^2)^{-(k-1)}$. However, for the multi-dimensional case, such analysis will become more nontrivial, and this becomes {one of obstacles} in the application of the CL method. When the probability density function does not decay sufficiently fast, biased result or even arithmetic overflow may occur. Recently, some numerical techniques to fix such issues have been proposed in \cite{attanasio2019dynamical, scherzer2020controlling}, which require further numerical analysis to understand their numerical error.

Without additional fixes, the CL method may still be valid for a wide range of observables if the probability density function is localized. Such existence usually depends on the parameters in the action. Unfortunately, as the parameter changes, the localization of the support may vanish in an unnoticeable way, which also introduces difficulty in judging the legitimacy of the numerical solution. Even worse, in the field theories, we can prove that such localized probability does not exist if the CL dynamics is not intervened. This will be detailed in the next section.

\section{Non-existence of localized probability in lattice field theories} \label{sec:theory}
In lattice field theories, the variables in the integral are a collection of group elements defined on the lattice. Specifically, we denote lattice nodes in the $(1+d)$-dimensional spacetime by the indices
\begin{equation*}
x = (t,x_1,\dots,x_d), \qquad t = 0,\dots,N_0-1, \quad x_1 = 0,\dots,N_1-1, \quad \dots, \quad x_d = 0,\dots,N_d-1.
\end{equation*}
Here we have assumed that all the lattice nodes are indexed by integers. For simplicity, we let $\mathcal{X} = (\mathbb{Z}/N_0\mathbb{Z}) \times (\mathbb{Z}/N_1\mathbb{Z}) \times \cdots \times (\mathbb{Z}/N_d\mathbb{Z})$ be the range of $x$, which also indicates the periodic boundary condition in our assumption. Below we are going to formulate the CL method for lattice field theories with general groups. We emphasize here that this is the first complete mathematical formulation of the CL method since it was proposed.

For each $x \in \mathcal{X}$ and each $\mu = 0,1,\dots,d$, we define a ``link variable'' $U_{x,\mu} \in G$, where $G$ is a compact Lie group with identity element $e$ and its Lie algebra being $\mathfrak{g}$. Let $\{U\} \in G^{(d+1)N_0 N_1 \cdots N_d}$ be the collection of all these link variables $U_{x,\mu}$. Then both the observable $O(\cdot)$ and the action $S(\cdot)$ are functions of $\{U\}$, and the expectation of the observable is given by
\begin{equation*}
\langle O \rangle = \frac{1}{Z} \int_{G^N} O(\{U\}) \exp\Big(-S(\{U\})\Big) \,d\{U\}, \qquad
Z = \int_{G^N} \exp\Big(-S(\{U\})\Big) \,d\{U\},
\end{equation*}
where the integral is defined by the Haar measure of $G$, and we have used the short-hand $N = (d+1)N_0 N_1 \cdots N_d$ for simplicity.

To apply the CL method to this problem when $S(\cdot)$ is complex, some addition assumptions need to be imposed:
\begin{description}
\item[\namedlabel{itm:A1}{(A1)}] The group $G$ is equipped with a Riemannian metric $\langle \cdot, \cdot \rangle_g$ for every $g \in G$, and the metric is bi-invariant.
\item[\namedlabel{itm:A2}{(A2)}] The group $G$ has a complexification
\begin{equation} \label{eq:Cartan}
G_{\mathbb{C}} = G \cdot \exp(i \mathfrak{g}),
\end{equation}
whose Lie algebra is $\mathfrak{g}_{\mathbb{C}} = \mathfrak{g} \oplus i \mathfrak{g}$.
\item[\namedlabel{itm:A3}{(A3)}] Both $O(\cdot)$ and $S(\cdot)$ can be extended to $G_{\mathbb{C}}^N$ as holomorphic functions.
\end{description}
By \ref{itm:A1}, we can assume that $\{X^1, X^2, \dots, X^m\}$ is an orthonormal basis of $\mathfrak{g}$ under the metric $\langle \cdot, \cdot \rangle_e$. The metric on $G_{\mathbb{C}}$ is chosen as the right-invariant metric:
\begin{equation} \label{eq:metric}
\begin{aligned}
  & \langle X_1 + i X_1', X_2 + i X_2' \rangle_e = \langle X_1, X_2 \rangle_e + \langle X_1', X_2' \rangle_e, \qquad \forall X_1, X_2, X_1', X_2' \in \mathfrak{g}, \\
  & \langle Z_1, Z_2 \rangle_g = \langle (dR_{g^{-1}})_g(Z_1), (dR_{g^{-1}})_g(Z_2) \rangle_e, \qquad \forall g \in G_{\mathbb{C}} \text{ and } Z_1, Z_2 \in T_g G_{\mathbb{C}},
\end{aligned}
\end{equation}
where $R_h: g \mapsto g h$ is the right translation operator, so that $(dR_h)_g$ is the map from the tangent space $T_g G_{\mathbb{C}}$ to the tangent space $T_{gh} G_{\mathbb{C}}$. Note that when $G$ is non-Abelian, this metric on $G_{\mathbb{C}}$ is in general not bi-invariant. The metric on $G_{\mathbb{C}}^N$ can then be naturally defined by summing up the metrics for all the components. Since each element in $\mathfrak{g}$ can be viewed as a right-invariant vector field on $G$, we will use the notation $\mathcal{L}_{X^a_{x,\mu}}$ to denote the Lie derivative with respect to the link variable $U_{x,\mu}$ along the right-invariant vector field $X^a$, and use $\mathcal{L}_{Y^a_{x,\mu}}$ to denote the Lie derivative with respect to the link variable $U_{x,\mu}$ along $Y^a = i X^a$. Thus by \ref{itm:A3}, we know that $S$ satisfies the Cauchy-Riemann equations
\begin{equation*}
\mathcal{L}_{X^a_{x,\mu}} \operatorname{Re}\, S = \mathcal{L}_{Y^a_{x,\mu}} \operatorname{Im}\, S, \qquad
\mathcal{L}_{X^a_{x,\mu}} \operatorname{Im}\, S = -\mathcal{L}_{Y^a_{x,\mu}} \operatorname{Re}\, S.
\end{equation*}
The equations for $O(\cdot)$ is similar. Let $K_{x,\mu}^a = -\mathcal{L}_{X_{x,\mu}^a} \operatorname{Re} S$ and $J_{x,\mu}^a = -\mathcal{L}_{X_{x,\mu}^a} \operatorname{Im} S$. We can then write down the CL equation:
\begin{equation} \label{eq:CL_QCD}
dU_{x,\mu} = \sum_{a=1}^m (dR_{U_{x,\mu}})_e \Big( \left[ K_{x,\mu}^a(\{U\}) \, dt + dw_{x,\mu}^a\right] X^a + \left[J_{x,\mu}^a(\{U\}) \, dt\right] Y^a \Big), \quad x \in \mathcal{X}, \quad \mu = 0,1,\dots,d,
\end{equation}
where the Brownian motions $w_{x,\mu}^a$ are independent of each other for different $x,\mu,a$, and we take the Stratonovich interpretation of the stochastic differential equation in \eqref{eq:CL_QCD}. Since $S$ is complex-valued, the link variable $U_{x,\mu}$ is generally in $G_{\mathbb{C}}$. Thus the CL method approximates the observable by
\begin{equation*}
\langle O \rangle \approx \frac{1}{N} \sum_{k=1}^N O\Big( \{U(T+k\Delta T)\} \Big)
\end{equation*}
for sufficiently large $T$ and sufficient time difference $\Delta T$. Numerically, the equation \eqref{eq:CL_QCD} is solved following
\begin{equation} \label{eq:scheme}
U_{x,\mu}(t+\Delta t) \approx \exp \left(
  \sum_{a=1}^m \left[ \left( \Delta t \, K_{x,\mu}^a(\{U(t)\})
  + \sqrt{2\Delta t} \,\eta_{x,\mu}^a \right) X^a + \Delta t \, J_{x,\mu}^a(\{U(t)\}) Y^a \right]
\right) U_{x,\mu}(t),
\end{equation}
where $\exp(\cdot)$ is the exponential map from $\mathfrak{g}_{\mathbb{C}}$ to $G_{\mathbb{C}}$, and each $\eta_{x,\mu}^a$ is a normally distributed random variable with mean zero and standard deviation one, generated at each time step.

In this presentation, $G$ can be regarded as the counterpart of the real axis, and then $G_{\mathbb{C}}$ is the counterpart of the complex plane. The verification of the CL method in the lattice field theory is similar to \cref{thm:CL}. We first define the non-negative probability density function as $P(\{U\}; t)$ for all $\{U\} \in G_{\mathbb{C}}^N$. The evolution equation of $P(\{U\}; t)$ is
\begin{equation} \label{eq:P_QCD}
\frac{\partial P}{\partial t} + \sum_{x\in \mathcal{X}} \sum_{\mu=0}^d \sum_{a=1}^m \left( \mathcal{L}_{X_{x,\mu}^a} (K_{x,\mu}^a P) + \mathcal{L}_{Y_{x,\mu}^a} (J_{x,\mu}^a P) \right) = \sum_{x \in \mathcal{X}} \sum_{\mu=0}^d \sum_{a=1}^m \mathcal{L}_{X_{x,\mu}^a} \mathcal{L}_{X_{x,\mu}^a} P,
\end{equation}
where
\begin{equation*}
K_{x,\mu}^a = -\mathcal{L}_{X_{x,\mu}^a} \mathrm{Re}\, S, \qquad
J_{x,\mu}^a = -\mathcal{L}_{X_{x,\mu}^a} \mathrm{Im}\, S.
\end{equation*}
Similarly to \eqref{eq:rho}, we define the complex-valued function $\rho(\{U\};t)$ for $\{U\} \in G^N$ as the solution of
\begin{equation} \label{eq:rho_QCD}
\frac{\partial \rho}{\partial t} + \sum_{x\in \mathcal{X}} \sum_{\mu=0}^d \sum_{a=1}^m \mathcal{L}_{X_{x,\mu}^a} [(K_{x,\mu}^a + i J_{x,\mu}^a) \rho] = \sum_{x\in \mathcal{X}} \sum_{\mu=0}^d \sum_{a=1}^m \mathcal{L}_{X_{x,\mu}^a} \mathcal{L}_{X_{x,\mu}^a} \rho.
\end{equation}
The initial condition of \eqref{eq:rho_QCD} is $\rho(\{U\};0) = p(\{U\})$, where $p(\{U\})$ is a probability density function on $G^N$. To describe the initial condition for \eqref{eq:P_QCD}, we need to use the Cartan decomposition \eqref{eq:Cartan}. In fact, the map $G \times \mathfrak{g} \rightarrow G_{\mathbb{C}}$ defined by \eqref{eq:Cartan} is a diffeomorphism. Therefore for every $U_{x,\mu} \in G_{\mathbb{C}}$, there exist unique $V_{x,\mu} \in G$ and $W_{x,\mu} \in \exp(i\mathfrak{g})$ such that $U_{x,\mu} = V_{x,\mu} W_{x,\mu}$. Thus we can define the initial condition of \eqref{eq:P_QCD} as
\begin{equation} \label{eq:P_init}
P(\{U\};t) = p(\{V\}) \prod_{x \in \mathcal{X}} \prod_{\mu=0}^d \delta_e(W_{x,\mu}),
\qquad \{U\} \in G_{\mathbb{C}}^N.
\end{equation}
where $\delta_e(\cdot)$ is the Dirac function defined on $\exp(i\mathfrak{g})$ whose value is infinity at the identity element. Now we are ready to state the theorem describing the validity of the CL method:
\begin{theorem}
Let $P(\{U\};t)$ be the {unique steady state} solution of \eqref{eq:P_QCD} with the initial condition \eqref{eq:P_init}, and $\rho(\{U\};t)$ be the solution of \eqref{eq:rho_QCD} with the initial condition $\rho(\{U\};0) = p(\{U\})$. We further suppose that $\mathcal{O}(\{U\}; t)$ with $\{U\} \in G_{\mathbb{C}}^N$ and $t\geq 0$ satisfies the backward Kolmogorov equation
\begin{equation*}
\frac{\partial \mathcal{O}}{\partial t} = \sum_{x\in \mathcal{X}} \sum_{\mu=0}^d \sum_{a=1}^m \left( K_{x,\mu}^a \mathcal{L}_{X_{x,\mu}^a} \mathcal{O} + J_{x,\mu}^a \mathcal{L}_{Y_{x,\mu}^a} \mathcal{O} \right) + \sum_{x \in \mathcal{X}} \sum_{\mu=0}^d \sum_{a=1}^m \mathcal{L}_{X_{x,\mu}^a} \mathcal{L}_{X_{x,\mu}^a} \mathcal{O}, \quad \mathcal{O}(\{U\}; 0) = O(\{U\}),
\end{equation*}
and it holds that
\begin{equation} \label{eq:H1}
\int_{G_{\mathbb{C}}^N} \mathcal{O}(\{U\}; \tau) P(\{U\}; t-\tau) \,d\{U\}
= \int_{G_{\mathbb{C}}^N} O(\{U\}) P(\{U\}; t) \,d\{U\}
\end{equation}
for any $t > 0$ and $\tau \in [0,t]$. Then for any $t > 0$,
\begin{equation} \label{eq:obs_QCD}
\int_{G^N} O(\{U\}) \rho(\{U\}; t) \,d\{U\} = \int_{G_{\mathbb{C}}^N} \mathcal{O}(\{U\}) P(\{U\}; t) \,d\{U\}.
\end{equation}
\end{theorem}

In the above theorem, the equality \eqref{eq:H1} corresponds to the assumption \ref{itm:H1} in \cref{sec:review}. The assumption \ref{itm:H3} is no longer needed due to the compactness of $G$. If we further assume that \eqref{eq:rho_QCD} has a steady-state solution
\begin{equation*}
\lim_{t\rightarrow +\infty} \rho(\{U\};t) = \rho_{\infty}(\{U\}) = \frac{1}{Z} \exp\Big(-S(\{U\}) \Big), \qquad \{U\} \in G^N,
\end{equation*}
then we can take the limit $t\rightarrow +\infty$ of \eqref{eq:obs_QCD} to validate the CL method. The proof of this theorem is completely parallel to the proof of \cref{thm:CL}, and we omit its details.

{Again, this theorem only gives us an unsatisfactory result due to the strong assumption \eqref{eq:H1}.} To ensure that \eqref{eq:H1} holds, we again need to have conditions similar to \eqref{eq:limits2}. Note that \eqref{eq:limits1} is not necessary again due to the compactness of $G$. The corresponding property holds for any observable $O$ only if the support of $P(\cdot;t)$ is compact for large $t$. Unfortunately, such localized probability density function does not exist, as will be proven in the following subsections. To begin with, we study a simple case where $G = U(1)$.

\subsection{Analysis for $U(1)$ theories} \label{sec:U1}
Due to its simplicity, the $U(1)$ theory is often employed to understand the properties of the CL method that are observed in other group theories \cite{aarts2010complex, salcedo2016does}. Here we also use this simple case to demonstrate our claims without involving the heavy notations in the group theory. When $G = U(1) = \{\exp(i\theta) \mid \theta \in \mathbb{R}\}$, its Lie algebra $\mathfrak{g}$ is the imaginary axis, and the metric can just be defined by
\begin{equation*}
\langle X, Y \rangle_e = X^{\dagger} Y, \qquad \forall X, Y \in \mathfrak{g},
\end{equation*}
where $\dagger$ denotes the complex conjugate. The complexification of $G$ is $G_{\mathbb{C}} = \{\exp(i\theta) \mid \theta \in \mathbb{C}\} = \mathbb{C} \backslash \{0\}$. Therefore the action $S$, as a function on $G_{\mathbb{C}}^N$, can be written as
\begin{displaymath}
S(e^{i\theta_1}, e^{i\theta_2}, \dots, e^{i\theta_N}) =
  S(e^{i(x_1+iy_1)}, e^{i(x_2 + iy_2)}, \dots, e^{i(x_N + iy_N)}),
\end{displaymath}
where we have assumed that $\theta_k = x_k + i y_k$ for $k = 1,\dots,N$. Let $\boldsymbol{x} = (x_1, x_2, \dots, x_N)^T$, $\boldsymbol{y} = (y_1, y_2, \dots, y_N)^T$ and define
\begin{displaymath}
  \bar{S}(\boldsymbol{x}, \boldsymbol{y}) = S(e^{i(x_1+iy_1)}, e^{i(x_2 + iy_2)}, \dots, e^{i(x_N + iy_N)}).
\end{displaymath}
Then $\bar{S}(\boldsymbol{x}, \boldsymbol{y})$ is periodic with respect to $x_1, x_2, \dots, x_N$, and the period is $2\pi$ for each variable. With such notations, the CL equation can be written similarly to \eqref{eq:sto}:
\begin{equation} \label{eq:CL_U1}
  \left\{
    \begin{array}{@{}lll}
      d\boldsymbol{x} = \boldsymbol{K}(\boldsymbol{x},\boldsymbol{y}) \,dt + d\boldsymbol{w},& \boldsymbol{K}(\boldsymbol{x},\boldsymbol{y}) = \mbox{Re}(-\nabla_{\boldsymbol{x}} \bar{S}(\boldsymbol{x}, \boldsymbol{y})),  \\[6pt]
      d\boldsymbol{y} = \boldsymbol{J}(\boldsymbol{x},\boldsymbol{y}) \,dt,& \boldsymbol{J}(\boldsymbol{x},\boldsymbol{y}) = \mbox{Im}(-\nabla_{\boldsymbol{x}} \bar{S}(\boldsymbol{x}, \boldsymbol{y})),
    \end{array}
  \right.
\end{equation}
where $\boldsymbol{w} = (w_1, \dots, w_N)^T$ with each $w_k$ being an independent Brownian motion. The initial condition satisfies $\boldsymbol{y}(0) = 0$. Now we are going to study \eqref{eq:CL_U1}, which has much simpler notations.

In order to localize the probability density function, we need to find a bounded, simply connected domain $\Omega \subset \mathbb{R}^N$ such that $\boldsymbol{J}(\boldsymbol{x}, \boldsymbol{y}) \cdot \boldsymbol{n}(\boldsymbol{y}) \leq 0$ for all $\boldsymbol{x} \in [0,2\pi)^N$ and $\boldsymbol{y} \in \partial \Omega$. Here $\boldsymbol{n}(\boldsymbol{y})$ denotes the outer unit normal vector of $\Omega$ at point $\boldsymbol{y}$. A two-dimensional case is illustrated in \cref{fig:su2local}. {\color{kyred} Note that due to the Brownian motion in the $\boldsymbol{x}$ direction, the domain $\Omega$ must be the same for every $\boldsymbol{x}$.} {Thus once the probability density function is completely attracted into $[0,2\pi)^N \times \overline{\Omega}$, it will forever be confined therein.} Here $\overline{\Omega}$ refers to the closure of $\Omega$. Since $\bar{S}$ is periodic with respect to $\boldsymbol{x}$ and $\boldsymbol{J}$ is the partial gradient of $\operatorname{Im} \bar{S}$ with respect to $\boldsymbol{x}$, by the divergence theorem, we have
\begin{equation} \label{eq:J_int}
\int_{[0,2\pi)^N} \boldsymbol{J}(\boldsymbol{x}, \boldsymbol{y}) \,d\boldsymbol{x} = \bm{0}
\end{equation}
for any $\boldsymbol{y}$. Since $\boldsymbol{n}$ only depends on $\boldsymbol{y}$, it follows that $\int_{[0,2\pi)^N} \boldsymbol{J}(\boldsymbol{x}, \boldsymbol{y})\cdot \boldsymbol{n}(\boldsymbol{y}) \,d\boldsymbol{x}  = 0$. Therefore $\boldsymbol{J}(\boldsymbol{x}, \boldsymbol{y}) \cdot \boldsymbol{n}(\boldsymbol{y}) \le 0$ for any $\boldsymbol{x} \in [0,2\pi)^N$ implies $\boldsymbol{J}(\boldsymbol{x}, \boldsymbol{y}) \cdot \boldsymbol{n}(\boldsymbol{y}) = 0$ for any $\boldsymbol{x} \in [0,2\pi)^N$. By Cauchy-Riemann equations, $\boldsymbol{J}(\boldsymbol{x}, \boldsymbol{y}) = \nabla_{\boldsymbol{y}} (\mathrm{Re}\, \bar{S})$. Therefore we conclude that the existence of localized probability density function requires
\begin{equation*}
\nabla_{\boldsymbol{y}} (\mathrm{Re}\, \bar{S}) \cdot \boldsymbol{n} = 0, \qquad \forall \boldsymbol{x} \in [0,2\pi)^N \text{ and }\boldsymbol{y} \in \partial \Omega.
\end{equation*}
The holomorphism of $\bar{S}$ indicates that $\mathrm{Re}\,\bar{S}$ is a harmonic function. Therefore its periodicity with respect to $\boldsymbol{x}$, together with the above Neumann boundary condition, shows that $\mathrm{Re}\,\bar{S}$ is a constant. As a consequence, $\bar{S}$ is a constant everywhere, as corresponds to the trivial case of the CL equation.

By the above analysis, we know that {for an irreducible action $S$ (meaning that the essential number of independent variables equals $N$),} if the probability density function is localized, it must be localized to {$[0,2\pi)^N \times \{\boldsymbol{y}_0\}$ for some $\boldsymbol{y}_0 \in \mathbb{R}^N$, corresponding to the case where $\bar{S}(\boldsymbol{x},\boldsymbol{y_0})$ is real, which is equivalent to a translated real action.} This is a pessimistic result, which indicates that when applying the CL method, one always needs to check whether the condition \eqref{eq:H1} holds, which is highly nontrivial since $P(\{U\};t)$ is not directly available. {\color{kyred}More precisely, due to the periodicity of $O(x)$, one can see from its Fourier expansion that $O(x + iy)$ is expected to grow at least exponentially in the imaginary direction. To cover this class of observables, we need the probability density function to have a decay rate faster than any exponential functions, which looks like a strong assumption. Such possibility will be discussed in our future works.} For a given observable, we refer the readers to \cite{scherzer2020controlling} for a recent work on the estimation of the boundary terms. The same phenomenon also exists for general group $G$, as will be detailed in the next subsection.

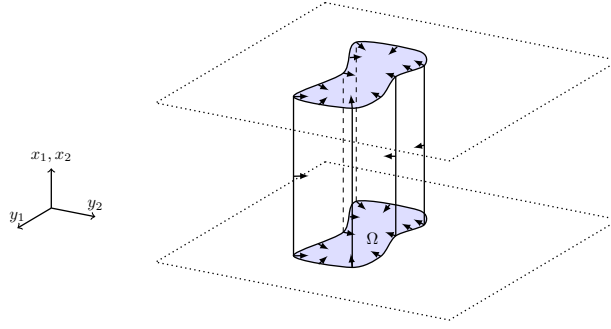
\begin{figure}
\centering
\resizebox{0.5\textwidth}{!}{%
\begin{tikzpicture}[scale=1.8, tdplot_main_coords,axis/.style={->,dashed},thick]
\draw[->] (3, -2, 0) -- (3.8, -2, 0) node [above] {$y_1$};
\draw[->] (3, -2, 0) -- (3, -1.4, 0) node [above] {$y_2$};
\draw[->] (3, -2, 0) -- (3, -2, 0.5) node [above] {$x_1,x_2$};
  
\coordinate  (d1) at (0,0,0){};
\coordinate  (d2) at (4,0,0){};
\coordinate  (d3) at (4,4,0){};
\coordinate  (d4) at (0,4,0){}; 
\coordinate  (d5) at (0,0,2){};
\coordinate  (d6) at (4,0,2){};
\coordinate  (d7) at (4,4,2){};
\coordinate  (d8) at (0,4,2){}; 

\draw [dotted] (d1)--(d2)--(d3)--(d4) -- (d1);                         
\draw [dotted] (d5)--(d6)--(d7)--(d8) -- (d5);

\draw[fill=blue!13] plot [smooth cycle] coordinates { (1,1.,0) (2,1.4,0) (3,1.3,0) (3,2.1,0) (2.5,2.2,0) (1.8,2.,0) (1.3,2.1,0) (0.9,1.8,0) } node at (2,1.8) {$\Omega$};
\draw[fill=blue!13] plot [smooth cycle] coordinates { (1,1.,2) (2,1.4,2) (3,1.3,2) (3,2.1,2) (2.5,2.2,2) (1.8,2.,2) (1.3,2.1,2) (0.9,1.8,2) };

\draw[dashed, thin] (1,1.,0)    -- (1,1.,2)   ;
\draw[dashed, thin] (2,1.4,0)   -- (2,1.4,2)  ;
\draw[] (3,1.3,0)   -- (3,1.3,2)  ;
\draw[] (3,2.1,0)   -- (3,2.1,2)  ;
\draw[] (1.8,2.,0)  -- (1.8,2.,2) ;
\draw[] (1.3,2.1,0) -- (1.3,2.1,2);

\draw[-latex, thick] (2.5,1.35,0) -- (2.6,1.55,0);
\draw[-latex, thick] (3.1,1.7 ,0) -- (2.8,1.65,0);
\draw[-latex, thick] (2.5,2.2 ,0) -- (2.4,2.0 ,0);
\draw[-latex, thick] (1.5,1.2 ,0) -- (1.45,1.315,0);
\draw[-latex, thick] (1.5,2.05,0) -- (1.4,1.85,0);
\draw[-latex, thick] (0.9,1.4 ,0) -- (1.2,1.475,0);

\draw[-latex, thick] (1  ,1   ,0) -- (1.2,1.25,0);
\draw[-latex, thick] (2  ,1.4 ,0) -- (2  ,1.57,0);
\draw[-latex, thick] (3  ,1.3 ,0) -- (2.9,1.45,0);
\draw[-latex, thick] (3  ,2.1 ,0) -- (2.6,1.85,0);
\draw[-latex, thick] (1.8,2.  ,0) -- (1.75,1.825,0);
\draw[-latex, thick] (1.3,2.1 ,0) -- (1.2,1.895,0);

\draw[-latex, thick] (2.5,1.35,2) -- (2.6,1.55,2);
\draw[-latex, thick] (3.1,1.7 ,2) -- (2.8,1.65,2);
\draw[-latex, thick] (2.5,2.2 ,2) -- (2.4,2.0 ,2);
\draw[-latex, thick] (1.5,1.2 ,2) -- (1.4,1.315 ,2);
\draw[-latex, thick] (1.5,2.05,2) -- (1.4,1.85,2);
\draw[-latex, thick] (0.9,1.5 ,2) -- (1.2,1.55,2);

\draw[-latex, thick] (1  ,1   ,2) -- (1.2,1.25,2);
\draw[-latex, thick] (2  ,1.4 ,2) -- (2  ,1.57,2);
\draw[-latex, thick] (3  ,1.3 ,2) -- (2.9,1.45,2);
\draw[-latex, thick] (3  ,2.1 ,2) -- (2.6,1.85,2);
\draw[-latex, thick] (1.8,2.  ,2) -- (1.75,1.825,2);
\draw[-latex, thick] (1.3,2.1 ,2) -- (1.2,1.895,2);

\draw[-latex, thick] (3  ,1.3 ,1) -- (2.9,1.45,1);
\draw[-latex, thick] (1.8,2.  ,1) -- (1.75,1.8,0.95);
\draw[-latex, thick] (1.3,2.1 ,1) -- (1.2,1.895,0.9);

\end{tikzpicture}
}
\caption{Localization of the distribution. The arrows denote the velocity field. \label{fig:su2local}}
\end{figure}

\subsection{Analysis for lattice gauge theories}
For lattice gauge theories, the derivation follows the same idea as the $U(1)$ theories. However, when $G$ is non-Abelian, the separation of the ``real part'' and the ``imaginary part'' becomes non-trivial. To this aim, we define two operators on $G_{\mathbb{C}}$ for any $g \in G_{\mathbb{C}}$:
\begin{equation*}
L_g: h \mapsto gh, \qquad \Psi_g: h \mapsto ghg^{-1},
\end{equation*}
where $L_g$ is the left translation operator, and $\Psi_g$ is known as the inner automorphism. It is obvious that $L_g = R_g \Psi_g$. Then the following proposition holds:
\begin{proposition} \label{prop:real_imag}
Let $U_{\tau}$ be a curve in $G_{\mathbb{C}}$ parametrized by $\tau \in \mathbb{R}$. Suppose $U_{\tau} = V_{\tau} W_{\tau}$ for $V_{\tau} \in G$ and $W_{\tau} \in \exp(i\mathfrak{g})$, and
\begin{equation} \label{eq:VW_tau}
\frac{dV_{\tau}}{d\tau} = (dR_{V_{\tau}})_e (X_{\tau}), \qquad
\frac{dW_{\tau}}{d\tau} = (dR_{W_{\tau}})_e (Y_{\tau}),
\end{equation}
for some $X_{\tau} \in \mathfrak{g}$ and $Y_{\tau} \in i\mathfrak{g}$. Then
\begin{equation} \label{eq:U_tau}
\frac{dU_{\tau}}{d\tau} = (dR_{U_{\tau}})_e \Big(X_{\tau} + (d\Psi_{V_\tau})_e Y_{\tau}\Big).
\end{equation}
Inversely, if \eqref{eq:U_tau} holds, then \eqref{eq:VW_tau} holds. Furthermore, if the vector $n_{\tau} \in T_{W_\tau} \exp(i\mathfrak{g})$ satisfies
\begin{equation*}
\left \langle \frac{dW_{\tau}}{d\tau}, n_{\tau} \right \rangle_{W_{\tau}} = 0.
\end{equation*}
Then
\begin{equation} \label{eq:normal}
\left \langle \frac{dU_{\tau}}{d\tau}, (dL_{V_\tau})_{W_{\tau}}(n_{\tau}) \right \rangle_{U_{\tau}} = 0.
\end{equation}
\end{proposition}

The proof of this proposition will be deferred to \cref{sec:proof2}. From this result, we can rewrite the complex Langevin equation \eqref{eq:CL_QCD} by separating the ``real'' and ``imaginary'' parts. The details are given in the following corollary:

\begin{corollary} \label{corr:CL_VW}
In the complex Langevin equation \eqref{eq:CL_QCD}, if $U_{x,\mu} = V_{x,\mu} W_{x,\mu}$ for $V_{x,\mu} \in G$ and $W_{x,\mu} \in \exp(i\mathfrak{g})$, then
\begin{equation*}
\begin{aligned}
dV_{x,\mu} &= \sum_{a=1}^m \Big( K_{x,\mu}^a(\{U\}) \,dt + dw_{x,\mu}^a \Big) (dR_{V_{x,\mu}})_e (X^a), \\
dW_{x,\mu} &= \sum_{a=1}^m \Big( J_{x,\mu}^a(\{U\}) \,dt \Big) (dR_{W_{x,\mu}})_e \Big( (d\Psi_{V_{x,\mu}^{-1}})_e (Y^a) \Big).
\end{aligned}
\end{equation*}
\end{corollary}

This corollary is a straightforward result of the equivalence between \eqref{eq:VW_tau} and \eqref{eq:U_tau}, and we omit its proof. It shows that the Brownian motion allows $V_{x,\mu}$ to explore everywhere in $G$. Therefore if $P(\{U\};t)$ has a compact support for all $t$, the support has the form $\Omega = G^N \cdot \Omega_I$, where $\Omega_I$ is a domain in $[\exp(i\mathfrak{g})]^N$. Like in the $U(1)$ theory, such $\Omega_I$ exists only when the action $S$ is a constant. To show this, we need the following lemma, which is the counterpart of \eqref{eq:J_int} in the $U(1)$ theory:

\begin{lemma} \label{lem:integral}
Let $f$ be a differentiable function on $G_{\mathbb{C}}$, and $n \in T_W \exp(i\mathfrak{g})$ for some $W \in \exp(i\mathfrak{g})$. Then
\begin{equation*}
\int_G \left\langle (dR_{VW})_e \left( \sum_{a=1}^m \mathcal{L}_{X^a} f(VW) Y^a\right), (dL_V)_W (n) \right\rangle_{VW} \,dV = 0,
\end{equation*}
where $\mathcal{L}_{X^a}$ denotes the Lie derivative along the right translationally invariant vector field generated by $X^a$.
\end{lemma}

\begin{proof}
Suppose $n = (dR_W)_e (Y)$ for some $Y \in i\mathfrak{g}$. Then
\begin{equation*}
(dL_V)_W (n) = (dL_V)_W \Big( (dR_W)_e (Y) \Big) = (dR_W)_V \Big( (dL_V)_e (Y) \Big).
\end{equation*}
According to the right translational invariance of the inner product, we have
\begin{equation*}
\left\langle (dR_{VW})_e \left( \sum_{a=1}^m \mathcal{L}_{X^a} f(VW) Y^a \right), (dL_V)_W (n) \right\rangle_{VW} =
  \left\langle (dR_V)_e \left( \sum_{a=1}^m \mathcal{L}_{X^a} f(VW) Y^a \right), (dL_V)_e (Y) \right\rangle_{V}.
\end{equation*}
By further assuming $Y = iX$ for $X \in \mathfrak{g}$ and using the definition of the metric \eqref{eq:metric}, we can rewrite the above expression as
\begin{equation*}
\left\langle (dR_{VW})_e \left(\sum_{a=1}^m \mathcal{L}_{X^a} f(VW) Y^a \right), (dL_V)_W (n) \right\rangle_{VW} =
  \left\langle (dR_V)_e \left( \sum_{a=1}^m \mathcal{L}_{X^a} \tilde{f}(V) X^a\right), (dL_V)_e (X) \right\rangle_{V},
\end{equation*}
where $\tilde{f}(V) = f(VW)$. Thus we can fixed $W$ and consider all terms on the right-hand side of the equation as objects on $G$. From this point of view, the right-hand side is in fact the Lie derivative of $\tilde{f}$ along the left-invariant vector field generated by $X \in \mathfrak{g}$. According to \cite[Theorem 14.35, Corollary 16.13]{lee2013introduction}, its integral equals zero.
\end{proof}

Now we are ready to state and prove our main result:
\begin{theorem}\label{thm:QCD-localized}
Suppose there exists a bounded, simply connected domain $\Omega_I \subset [\exp(i\mathfrak{g})]^N$ such that for $\Omega = G^N \cdot \Omega_I \subset G_{\mathbb{C}}^N$,
\begin{equation}\label{eq:QCD-localize-condition}
  \langle Z(\{U\}), n(\{U\}) \rangle_{\{U\}} \leq 0, \qquad \forall \{U\} \in \partial \Omega,
\end{equation}
where $n$ denotes the outer unit normal of $\Omega$, and $Z(\{U\})$ is the velocity field
\begin{equation}
  Z(\{U\}) = \bigoplus_{x,\mu} \sum_{a=1}^m (dR_{U_{x,\mu}})_e \Big( K_{x,\mu}^a(\{U\}) X^a + J_{x,\mu}^a(\{U\}) Y^a \Big).
\end{equation}
Then ${S}$ is a constant everywhere.
\end{theorem}
\begin{proof}
For $\{W\} \in \partial \Omega_I$, we can write the outer unit normal of $\Omega_I$ in the following form:
\begin{equation*}
n(\{W\}) = \bigoplus_{x,\mu} n_{x,\mu}(\{W\}), \qquad n_{x,\mu}(\{W\}) \in T_{W_{x,\mu}} \exp(i\mathfrak{g}).
\end{equation*}
Since $\partial \Omega = G^N \cdot \partial \Omega_I$, for any $\{U\} \in \partial \Omega$, we can find $\{V\} \in G^N$ and $\{W\} \in [\exp(i\mathfrak{g})]^N$ such that $U_{x,\mu} = V_{x,\mu} W_{x,\mu}$. Then according to \eqref{eq:normal}, the outer normal vector of $\Omega$ at $\{U\} \in \partial \Omega$ is
\begin{equation*}
n_{\{U\}} = \bigoplus_{x,\mu} (dL_{V_{x,\mu}})_{W_{x,\mu}} \Big(n_{x,\mu}(\{W\}) \Big),
\end{equation*}
and thus
\begin{equation*}
\langle Z(\{U\}), n(\{U\}) \rangle_{\{U\}} =
  \sum_{x,\mu} \left\langle (dR_{U_{x,\mu}})_e \left( \sum_{a=1}^m J_{x,\mu}^a(\{U\}) Y^a \right), (dL_{V_{x,\mu}})_{W_{x,\mu}} \Big(n_{x,\mu}(\{W\}) \Big) \right\rangle_{U_{x,\mu}}.
\end{equation*}
Since $J_{x,\mu}^a(\{U\})$ denotes the derivative of $-\operatorname{Im} S$ along $R_{V_{x,\mu}}(X^a)$, we see from \cref{lem:integral} that
\begin{equation*}
\int_{G^N} \langle Z(\{U\}), n(\{U\}) \rangle_{\{U\}} \,d\{V\} = 0.
\end{equation*}
By \eqref{eq:QCD-localize-condition}, the value of $\langle Z(\{U\}), n(\{U\}) \rangle_{\{U\}}$ must be zero for every $\{U\} \in \partial \Omega$, which can be considered as the homogeneous Neumann boundary condition of $\operatorname{Re} S$ on $\partial \Omega$ due to the fact that $J_{x,\mu}^a$ can also be regarded as the derivative of $\operatorname{Re} S$ along $R_{U_{x,\mu}} (Y^a)$, and $K_{x,\mu}^a$ plays no role in the inner product $\langle Z(\{U\}), n(\{U\}) \rangle_{\{U\}}$. Furthermore, the holomorphism of $S(\{U\})$ indicates that $\operatorname{Re} S$ is harmonic. Therefore by the uniqueness of the solutions to elliptic equations \cite[Proposition 7.6]{taylor2011partial}, we know that the real part of $S$ is a constant, and thus $S$ is also a constant.
\end{proof}

\section{Localized probability density functions with gauge cooling technique} \label{sec:gc}
The analysis in \cref{sec:theory} reveals the reason why the application of the CL method is highly delicate. As mentioned in \cref{sec:introduction}, the method became successful mainly after the method of gauge cooling was proposed \cite{seiler2013gauge}. {Before this work, the idea of gauge cooling already exists in some literature \cite{berges2008real, aarts2013stablity}, known as gauge fixing, which is used to study some simple models. In \cite{cai2020does}, the study of gauge fixing is extended to the analysis of gauge cooling for} the one-dimensional $SU(n)$ theory, and it is found that localized probability density functions can sometimes be obtained after applying the gauge cooling technique. However, this is not guaranteed by gauge cooling, and in \cite{cai2020does}, the authors also showed some cases where the probability density functions are global even after gauge cooling is applied. In such cases, it appears in the numerical experiments in \cite{cai2020does} that some legitimate results can also be generated. According to our observation in \cref{sec:toy_model}, such results may also be biased but only with small deviations from the exact integrals. In this work, we will restudy these examples and get a deeper understanding of the gauge cooling technique. To begin with, we will briefly review how gauge cooling works in the lattice field theory.

\subsection{Review of the gauge cooling technique}
The gauge cooling technique is developed based on the gauge invariance of the lattice field theory. For the complexified gauge field $\{U\} \in G_{\mathbb{C}}^N$, we introduce the gauge transformation by
\begin{equation} \label{eq:gauge_transformation}
\widetilde{U}_{x,\mu} = g_x^{-1} U_{x,\mu} g_{x+\hat{\mu}}, \qquad \forall x\in \mathcal{X}, \quad \forall \mu = 0,1,\cdots,d,
\end{equation}
where $g_x \in G_{\mathbb{C}}$ is defined for every $x \in \mathcal{X}$, and $\hat{\mu}$ refers to the canonical unit vector $(0,\cdots,0,1,0,\cdots,0)^T$ whose $\mu$th component is $1$. Here we remind the readers that the periodic boundary condition is used so that $x+\hat{\mu}$ is always well defined. After the transformation, a new field $\{\widetilde{U}\}$ is formed by the link variables $\widetilde{U}_{x,\mu}$. In the gauge theory, both the observable and the action are invariant under gauge transformation:
\begin{equation*}
O(\{\widetilde{U}\}) = O(\{U\}), \qquad S(\{\widetilde{U}\}) = S(\{U\}).
\end{equation*}
As a result, we can apply any gauge transformation at any time during the evolution of the CL equation, which does not introduce any biases. Therefore, after each time step \eqref{eq:scheme}, we can choose suitable $g_x \in G_{\mathbb{C}}$ for each $x \in \mathcal{X}$ and apply the gauge transformation \eqref{eq:gauge_transformation} to $\{U(t)\}$ so that the dynamics can hopefully be stabilized. To choose $g_x$ appropriately, we let $F(\{U\})$ for all $\{U\} \in G_{\mathbb{C}}^N$ be the distance between $\{U\}$ and the submanifold $G^N$, and then solve the following minimization problem:
\begin{equation} \label{eq:minimization}
\argmin_{g_x \in G_{\mathbb{C}} \text{ for all } x \in \mathcal{X}} F(\{\widetilde{U}\}).
\end{equation}
Gauge cooling refers to the gauge transformation \eqref{eq:gauge_transformation} using the solution of the optimization problem \eqref{eq:minimization}. We hope that such a choice of $g_x$ can help pull the sample field closer to $G^N$, causing a faster decay of the probability density function, so that the boundary terms can be reduced or eliminated. A formal justification of this method can be found in \cite{nagata2016justification}, which shows that the gauge cooling is unbiased. Numerically, the optimization problem \eqref{eq:minimization} is solved by gradient descent method with a sufficient number of iterations \cite{seiler2013gauge, aarts2013controlling}.

Once $F(\cdot)$ is chosen, we can define the submanifold
\begin{equation*}
M = \Big\{ \{U\} \,\Big|\, F(\{U\}) \leq F(\{\widetilde{U}\}) \text{ for any } g_x, \, x \in \mathcal{X} \Big\},
\end{equation*}
which is the set of fields with ``optimal guage'' that minimizes the distance to $G^N$. The gauge cooling technique ensures that the field stays on $M$ in the CL method. Therefore the governing equation can be formulated by mapping the right-hand side of \eqref{eq:CL_QCD} to the tangent space of $M$, and this mapping is linear but depends on the choice of the distance $F$.

By restricting the dynamics on the manifold $M$, we may have a chance to localize the probability density function since the argument in the previous section no longer holds. Deeper analysis of the gauge cooling technique requires the detailed form of $F(\{U\})$. One general choice of the distance function is
\begin{equation} \label{eq:dis}
F(\{U\}) = \frac{1}{2} \sum_{x \in \mathcal{X}} \sum_{\mu=0}^d  \langle Y_{x,\mu}, Y_{x,\mu} \rangle_e.
\end{equation}
Here $Y_{x,\mu} \in i\mathfrak{g}$ appears in the Cartan decomposition $U_{x,\mu} = V_{x,\mu} \exp(Y_{x,\mu})$. Interestingly, with such a distance function, we can find that gauge cooling does not have any effect when $G$ is an Abelian group, including the simplest $U(1)$ theory. The result will be proven in the following subsection.

\subsection{Gauge cooling for Abelian groups}
Gauge cooling takes effect only when the gauge field does not lie on the manifold $M$. However, when $G$ is Abelian (so that $G_{\mathbb{C}}$ is also Abelian), the CL dynamics automatically ensures that the field always stays on $M$. The details are given in the following theorem:

\begin{theorem}
Suppose $G$ is an Abelian group and the distance function $F(\{U\})$ is defined by \eqref{eq:dis}. Then the CL equation \eqref{eq:CL_QCD} guarantees that $\{U(t)\} \in M$ for all $t$ if the initial condition $\{U(0)\} \in M$.
\end{theorem}

\begin{proof}
Since $G$ is an Abelian group, we can assume that $g_x = \exp(h_x)$ and $U_{x,\mu} = V_{x,\mu} \exp(Y_{x,\mu})$ for $h_x, Y_{x,\mu} \in i\mathfrak{g}$, so that the gauge transformation \eqref{eq:gauge_transformation} becomes
\begin{equation} \label{eq:gc_abelian}
\widetilde{U}_{x,\mu} = V_{x,\mu} \exp(Y_{x,\mu} - h_x + h_{x+\hat{\mu}}), \qquad \forall x \in \mathcal{X}, \quad \forall \mu = 0,1,\dots,d.
\end{equation}
Here it suffices to choose $g_x \in \exp(i\mathfrak{g})$ since adding a factor in $G$ does not change the distance function $F(\{\widetilde{U}\})$. Thus the distance function turns out to be
\begin{equation*}
F(\{\widetilde{U}\}) = \frac{1}{2} \sum_{x \in \mathcal{X}} \sum_{\mu=0}^d \langle Y_{x,\mu} - h_x + h_{x+\hat{\mu}}, Y_{x,\mu} - h_x + h_{x+\hat{\mu}} \rangle_e.
\end{equation*}
This is a quadratic function in the linear space $i\mathfrak{g}$, so that the minimization problem \eqref{eq:minimization} can be solved by solving the first-order optimality condition:
\begin{equation} \label{eq:Poisson}
\sum_{\mu=0}^d (2h_x - h_{x-\hat{\mu}} - h_{x+\hat{\mu}})
  = \sum_{\mu=0}^d (Y_{x,\mu} - Y_{x-\hat{\mu},\mu}),
  \qquad \forall x \in \mathcal{X}.
\end{equation}
This equation has the form of a discrete Poisson equation with periodic boundary condition, and therefore the solution is unique up to a constant, which does not change the gauge transformation. This also indicates that the manifold $M$ can be described by
\begin{equation*}
M = \left\{ \{U\} \Bigg| \prod_{\mu=0}^d U_{x,\mu} U_{x-\hat{\mu},\mu}^{-1} \in G \text{ for all } x \in \mathcal{X} \right\},
\end{equation*}
since $\{U\} \in M$ indicates that the right-hand side of \eqref{eq:Poisson} is zero.

We assume that the initial value of $\{U(0)\}$ lies on $M$. Since $G$ is Abelian, the inner automorphism $\Psi_g$ is the identity operator. According to \cref{corr:CL_VW}, we can derive the following equation for $Y_{x,\mu}$:
\begin{equation*}
dY_{x,\mu} = \sum_{a=1}^m J_{x,\mu}^a(\{U\}) Y^a \,dt.
\end{equation*}
If we can show that
\begin{equation} \label{eq:optimal_gauge}
\sum_{\mu=0}^d (J_{x,\mu}^a - J_{x-\hat{\mu},\mu}^a) = 0,
  \qquad \forall x \in \mathcal{X}, \quad \forall a = 1,\dots,m,
\end{equation}
then it is clear that the right-hand side of \eqref{eq:Poisson} will remain zero if its initial value is zero. The equations \eqref{eq:optimal_gauge} can be seen from the gauge invariance of the action $S(\{U\})$. For given $x \in \mathcal{X}$ and $a = 1,\cdots,m$, consider the gauge transformation
\begin{equation*}
\tilde{U}_{x,\mu}(\tau) = \exp(-\tau X^a) U_{x,\mu}, \qquad \tilde{U}_{x-\hat{\mu},\mu}(\tau) = U_{x-\hat{\mu},\mu} \exp( \tau X^a), \qquad \forall \mu = 0,1,\dots,d,
\end{equation*}
and other components of $\{\widetilde{U}\}$ stay unchanged. Let $\widetilde{S}(\tau) = S\left(\{\widetilde{U}(\tau)\}\right)$ for $\tau \in \mathbb{R}$. By chain rule, it is straightforward to verify that
\begin{equation*}
\operatorname{Im} \frac{d\widetilde{S}}{d\tau} = \sum_{\mu=0}^d (J_{x,\mu}^a - J_{x-\hat{\mu},\mu}^a).
\end{equation*}
The gauge invariance of $S$ shows that $\widetilde{S}$ is a constant. Therefore the above derivative is zero, meaning that \eqref{eq:optimal_gauge} holds.
\end{proof}

The above theorem shows that for the exact CL dynamics, gauge cooling does not change the field. However, {this only refers to the continuous case, where the stochastic differential equation can be exactly solved. Numerically,} after time discretization, the field may deviate from the manifold $M$. In this case, the gauge cooling technique can act as a projection to keep the field on $M$, {which also helps stabilize the dynamics. This is used in a recent work \cite{hirasawa2020complex}, where the $U(1)$ gauge theory is considered. In fact, even in the real Langevin dynamics, in which we are sure that gauge cooling has no effect, applying such technique also helps avoid the possible instability due to computer arithmetic.} Next, we are going to focus on the $SU(n)$ theory, which is non-Abelian so that gauge cooling is expected to be effective.

\begin{remark}
The distance function used in \cite{hirasawa2020complex} is
\begin{equation*}
F(\{U\}) = \sum_{x,\mu} \left[\exp \left(2\sqrt{\langle Y_{x,\mu}, Y_{x,\mu} \rangle_e} \right) + \exp \left(-2\sqrt{\langle Y_{x,\mu}, Y_{x,\mu} \rangle_e}\right) - 2 \right],
\end{equation*}
which differs slightly from our definition \eqref{eq:dis}. We expect that these two functions have similar effects since their leading-order term agrees after Taylor expansion. Further studies are needed to understand the effects of different distance functions.
\end{remark}

\subsection{One-dimensional $SU(n)$ theory} \label{sec:Polyakov}
Now we focus on the $SU(n)$ theory, which is non-Abelian for all $n > 1$, and is most commonly used in lattice QCD. When $G = SU(n)$, its Lie algebra $\mathfrak{g}$ is the space of all traceless skew-Hermitian matrices of order $n$, whose dimension $m = n^2-1$. The metric on $SU(n)$ is defined by
\begin{equation*}
\langle X_1, X_2 \rangle_U = \frac{1}{2} \mathrm{tr} (X_1^{\dagger} X_2), \qquad \forall U \in SU(n), \quad \forall X_1,X_2 \in T_U G,
\end{equation*}
and the complexification of $SU(n)$ is the special linear group $SL(n,\mathbb{C})$. In this case, choosing the distance function as \eqref{eq:dis} may be inconvenient since the calculation of Cartan decomposition is not straightforward. Therefore we follow \cite{seiler2018status} which defines $F(\{U\})$ by
\begin{equation} \label{eq:distance}
F(\{U\}) = \sum_{x \in \mathcal{X}} \sum_{\mu=0}^d [\mathrm{tr}(U_{x,\mu}^{\dagger} U_{x,\mu}) - n].
\end{equation}
It can be shown that the above quantity equals zero only when all $U_{x,\mu}$ are in $SU(n)$. In what follows, we will study the one-dimensional case ($d = 0$) as an extension of the work in \cite{cai2020does}, which also gives some insight of the multi-dimensional case, as will be commented at the end of this section.

For $d = 0$ with periodic boundary conditions, according to the study in \cite{cai2020does}, the manifold $M$ can be characterized as
\begin{equation} \label{eq:manifold}
M = \{ (\lambda_1, \lambda_2, \dots, \lambda_n) \in \mathbb{C}^n \mid \lambda_1 \lambda_2 \cdots \lambda_n = 1\}.
\end{equation}
{In fact, the one-dimensional case is highly similar to the one-link case, whose formulation has been given in \cite{aarts2013stablity}. In the case of $N$ links,} the stochastic differential equation (defined by It\^o calculus) on $M$ has been derived in \cite[Eq. (4.9)]{cai2020does}:\footnote{Compared with the notations in \cite{cai2020does}, we have changed the definition of $w^a$ so that $w^a$ defined in \eqref{eq:KJw} corresponds to the standard Brownian motion. Therefore, compared with equation (4.9) in \cite{cai2020does}, an additional coefficient $1/\sqrt{N}$ is seen in front of the first term on the right-hand side of \eqref{eq:dlambda}.}
\begin{equation} \label{eq:dlambda}
\frac{1}{N} d\lambda_j = -\frac{1}{\sqrt{N}} \sum_{a=1}^m \lambda_j X^a_{jj} \,dw^a - \left[
  \sum_{a=1}^m [(K^a + i J^a) X^a_{jj} + 2 e_j^T X^a \Omega_j X^a e_j] + \frac{2(n^2-1)}{n}
\right] \lambda_j \,dt, \quad j = 1,\dots,n.
\end{equation}
Here $X^1, \dots, X^m$ are the orthonormal basis of $\mathfrak{g}$, and $X_{jj}^a$ is the $j$th diagonal element of $X^a$. The quadratic Casimir invariant is
\begin{equation} \label{eq:Casimir}
-\sum_{a=1}^m X^a X^a = \frac{2(n^2-1)}{n} I,
\end{equation}
where $I$ is the identity matrix. For example, when $n=2$, the basis can be chosen as $X^1 = i \sigma_x$, $X^2 = i \sigma_y$, $X^3 = i \sigma_z$, where $\sigma_{x,y,z}$ are Pauli matrices; similarly, when $n=3$, the basis can be chosen as $i$ times Gell-Mann matrices. In \eqref{eq:dlambda}, the matrix $\Omega_j$ is
\begin{displaymath}
\Omega_j = \mathrm{diag} \left\{ \frac{\lambda_1}{\lambda_1 - \lambda_j}, \dots, \frac{\lambda_{j-1}}{\lambda_{j-1} - \lambda_j}, 0, \frac{\lambda_{j+1}}{\lambda_{j+1} - \lambda_j}, \dots, \frac{\lambda_N}{\lambda_N - \lambda_j} \right\},
\end{displaymath}
and $K^a, J^a$ and $w^a$ are defined by
\begin{equation} \label{eq:KJw}
K^a = \frac{1}{N} \sum_{x \in \mathcal{X}} K_{x,0}^a, \qquad
J^a = \frac{1}{N} \sum_{x \in \mathcal{X}} J_{x,0}^a, \qquad
w^a = \frac{1}{\sqrt{N}} \sum_{x \in \mathcal{X}} w_{x,0}^a.
\end{equation}
Without repeating the details of the derivation in \cite{cai2020does}, we just mention here that the equation \eqref{eq:dlambda} is derived by solving the minimization problem \eqref{eq:minimization} analytically, and then couple the solution into the CL dynamics.

We would now like to study whether it is possible to localize the probability density function on $M$. To clarify the effect of gauge cooling, we first assume that $J^a = 0$, which singles out the effect of gauge cooling. In this case, we have the following result:
\begin{theorem}
If $J^a = 0$ for all $a = 1,\ldots,m$ in \eqref{eq:dlambda}, we have
\begin{displaymath}
\frac{d}{dt} \sum_{j=1}^n |\lambda_j|^2 \leq 0
\end{displaymath}
for any initial value $\lambda_1(0), \dots, \lambda_n(0)$.
\end{theorem}
\begin{proof}
By replacing the term $2(n^2-1)/n$ in \eqref{eq:dlambda} using \eqref{eq:Casimir}, we can rewrite \eqref{eq:dlambda} as
\begin{equation*}
d\lambda_j = N\left[ -\frac{1}{\sqrt{N}} \sum_{a=1}^m \lambda_j X^a_{jj} \,dw^a - \sum_{a=1}^m \left(K^a X_{jj}^a + e_j^T X^a D_j X^a e_j - \sum_{a=1}^m (X_{jj}^a)^2 \right) \lambda_j \,dt\right],
\end{equation*}
where $D_j = \mathrm{diag} \left\{ \frac{\lambda_1+\lambda_j}{\lambda_1 - \lambda_j}, \dots, \frac{\lambda_{j-1}+\lambda_j}{\lambda_{j-1} - \lambda_j}, 0, \frac{\lambda_{j+1}+\lambda_j}{\lambda_{j+1} - \lambda_j}, \dots, \frac{\lambda_N+\lambda_j}{\lambda_N - \lambda_j} \right\}$. Since $X^a$ is skew-Hermitian, it follows that
\begin{equation*}
d\bar{\lambda}_j = N\left[ \frac{1}{\sqrt{N}} \sum_{a=1}^m \bar{\lambda}_j X^a_{jj} \,dw^a - \sum_{a=1}^m \left( -K^a X_{jj}^a + e_j^T X^a \bar{D}_j X^a e_j - \sum_{a=1}^m (X_{jj}^a)^2 \right) \bar{\lambda}_j \,dt\right].
\end{equation*}
Therefore by It\^o calculus,
\begin{equation*}
d(|\lambda_j|^2) = \bar{\lambda}_j \,d\lambda_j + \lambda_j \,d\bar{\lambda}_j - 2N \sum_{a=1}^m |\lambda_j|^2 (X_{jj}^a)^2 \,dt = -2N \sum_{a=1}^m e_j^T X^a (\mathrm{Re}\, D_j) X^a e_j |\lambda_j|^2  \,dt.
\end{equation*}
Summing up the above equation for all $j = 1,\dots,n$, we obtain
\begin{equation*}
\begin{split}
    \frac{d}{dt} \sum_{j=1}^n |\lambda_j|^2 &
    =  -2N \sum_{j=1}^n  \sum_{a=1}^m e_j^T X^a (\operatorname{Re} D_j) X^a e_j |\lambda_j|^2
    = 2N \sum_{j=1}^n\sum_{k\ne j}^n \sum_{a=1}^m  |\lambda_j|^2 |X_{jk}|^2 
        \operatorname{Re}\frac{\lambda_{k}+\lambda_{j}}{\lambda_{k}-\lambda_{j}} \\
    & = 2N \sum_{j=1}^n\sum_{k\ne j}^n \sum_{a=1}^m  |\lambda_k|^2  |X_{kj}|^2 
        \operatorname{Re} \frac{\lambda_{j}+\lambda_{k}}{\lambda_{j}-\lambda_{k}}
    = N \sum_{j=1}^n\sum_{k\ne j}^n \sum_{a=1}^m  |X_{jk}|^2  (|\lambda_j|^2-|\lambda_k|^2) 
        \operatorname{Re} \frac{\lambda_{k}+\lambda_{j}}{\lambda_{k}-\lambda_{j}} \\
    & = N \sum_{j=1}^n\sum_{k\ne j}^n \sum_{a=1}^m  |X_{jk}|^2  (|\lambda_j|^2-|\lambda_k|^2) 
        \frac{|\lambda_{k}|^2-|\lambda_{j}|^2}{|\lambda_{k}-\lambda_{j}|^2} \leq 0,
\end{split}
\end{equation*}
which completes the proof.
\end{proof}

This theorem indicates that gauge cooling has introduced additional velocity which helps confine the probability density function. It can then be expected that when $J^a$ is small, the probability density function can be localized. A special case for the $SU(2)$ theory with $N=1$ and $S(U) = -(A+iB) \operatorname{tr} U$ for $A,B\in \mathbb{R}$ has been considered {in a number of references including \cite{berges2008real, aarts2013stablity, cai2020does}}. When $n=2$, the manifold $M$ defined in \eqref{eq:manifold} is identical to $\mathbb{C} \backslash \{0\}$, which turns out to be similar to the complexified $U(1)$ one-link theory. Mimicking the transformations in \cref{sec:U1}, we can write down the CL dynamics as
\begin{equation} \label{eq:SU2_sde}
  \begin{aligned}
    dx &= K\,dt + dw, & K &= 2\left(-A \cosh y \sin x+B \sinh y \cos x+\frac{\sin 2 x}{\cosh 2 y-\cos 2 x}\right),\\
    dy &= J\,dt, & J &= -2\left(A \sinh y \cos x+B \cosh y \sin x+\frac{\sinh 2 y}{\cosh 2 y-\cos 2 x}\right),
  \end{aligned}
\end{equation}
where $x$ is considered to be periodic with period $2\pi$. Let $z = x + iy$, then the expectation value of interest is $O(z) = e^{iz}+e^{-iz}$. In Theorem 4.1 of \cite{cai2020does}, it was shown that when $(A, B)$ locates in a certain region, the probability function is localized and the CL method produces the correct result. For example in their numerical experiments with $A=1$, when $B = 0.2$, $(A,B)$ is inside the region, and the CL result gives correct results for all observables. However, when $A, B$ are chosen such that the support of probability density function is not compact, the reference \cite{cai2020does} also shows that the CL method may produce a result that is close to the exact integral, but it is unclear whether the error comes from the bias or the stochastic noise. More precisely, three cases $(A,B) = (1,2), (5,1)$ and $(5,10)$ are tested in \cite{cai2020does}, among which only the case $(A,B) = (1,2)$ shows a clear bias, while no conclusion is drawn for the other two sets of parameters.

By our analysis of the model problem in \cref{sec:toy_model}, it can be expected that for $A = 5$, the results may also be biased due to the global support of the probability density function. To confirm this, we carry out the analysis of the decay rate in the similar way to \cref{sec:decay}. Assume the steady-state probability function generated by the stochastic differential equation \eqref{eq:SU2_sde} is of the form
\begin{equation*}
  P(x,y) \approx c(x) e^{-\beta y},
\end{equation*}
at large $y > 0$. Substituting this ansatz into the associated FP equation and considering only the leading-order terms, we obtain
\begin{equation*}
    L^T P \approx e^{y-\beta  y} \left(c'(x) (B \cos (x)-A \sin (x))+(\beta -2) c(x) (A \cos (x)+B \sin (x))\right).
\end{equation*}
Equating this expression to zero, we can solve $c(x)$ as
\begin{equation*}
    c(x) = C (B \cos (x)-A \sin (x))^{\beta -2}, \quad C \in \mathbb{R}.
\end{equation*}
Due to the fact that $P(x,y)$ must be positive, the parameter $\beta$ can only take the value $2$. To check whether the solution is biased, we follow the condition \eqref{eq:limits2} to check the growth rate of $J$ and $O$. When $y$ is large, $J(x,y) \propto e^y$, $O(x+iy) \propto e^y$, whose product exactly cancels the decay of $P(x,y)$, meaning that the CL method fails to produce unbiased result when $P(x,y)$ is not localized.

The behavior of the one-dimensional $SU(2)$ theory turns out to be very similar to the model problem studied in \cref{sec:toy_model}: when the parameter exceeds a certain threshold, the distribution of samples is no longer localized, and then biased result is produced. To verify this, we again use the numerical method introduced in \cref{sec:toy1-nm} to solve the FP equation, when $A=1$, the marginal probability density function $P_y(y)$ is plotted in \cref{fig:SU2PyB}, one can clearly observe that when $B$ increases, the support of the probability density function turns global at a certain point, and the decay rate agrees with our theoretical analysis. Based on the computed probability density function, we compute the observables and show the results in \cref{fig:su2vsA1}, from which one can observe that the CL results deviate from the exact values smoothly. We have also done the same numerical tests for $A = 5$, for which the support of the probability density function is the whole domain for any $B > 0$. The results shown in \cref{fig:su2vsA5} confirm the existence of the bias for all $A = 5$ and $B > 0$.

\begin{figure}[htbp]
\centering
\includegraphics[width=0.35\textwidth]{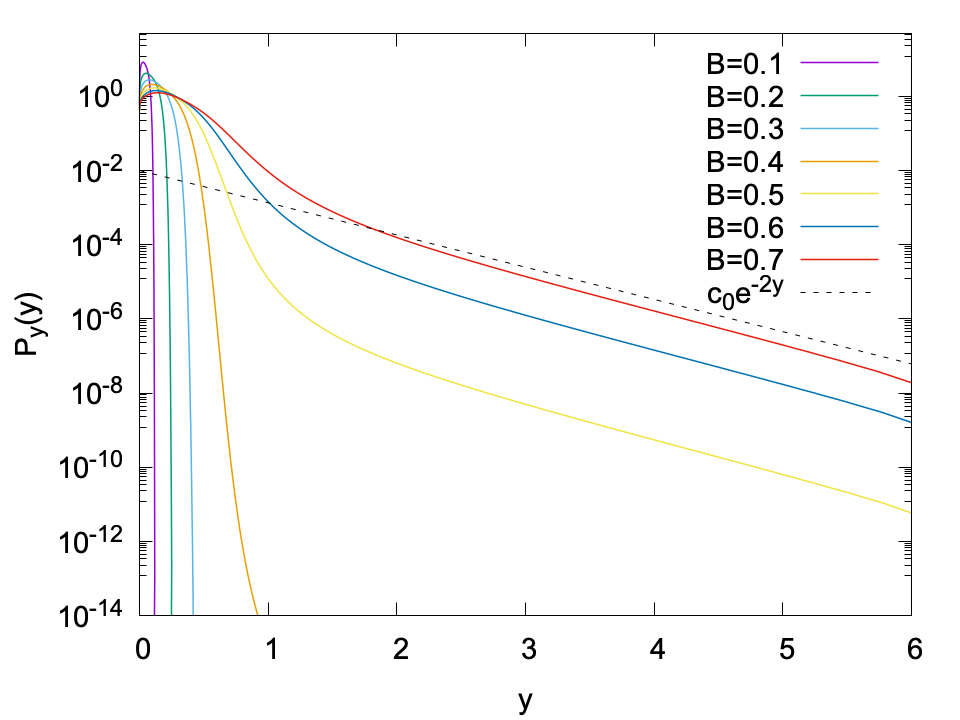}\qquad
\caption{The decay of $P_y$ for $A=1$. \label{fig:SU2PyB}}
\end{figure}

\begin{figure}[htbp]
  \centering
  \subfigure[$A=1$]
  {\includegraphics[width=0.35\textwidth]{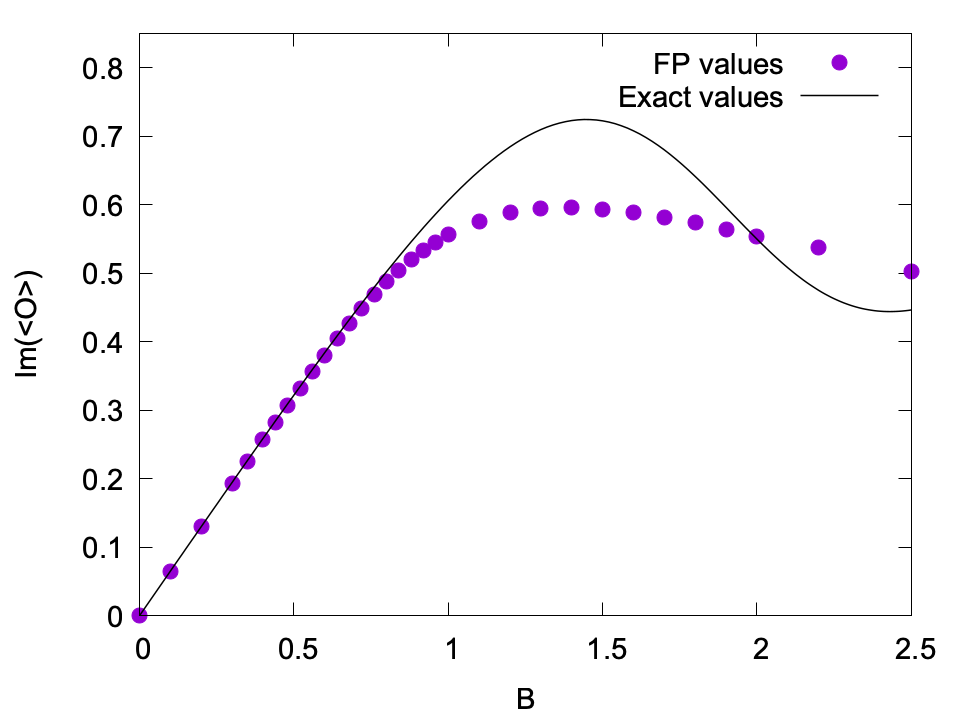}\label{fig:su2vsA1}}\qquad
  \subfigure[$A=5$]
  {\includegraphics[width=0.35\textwidth]{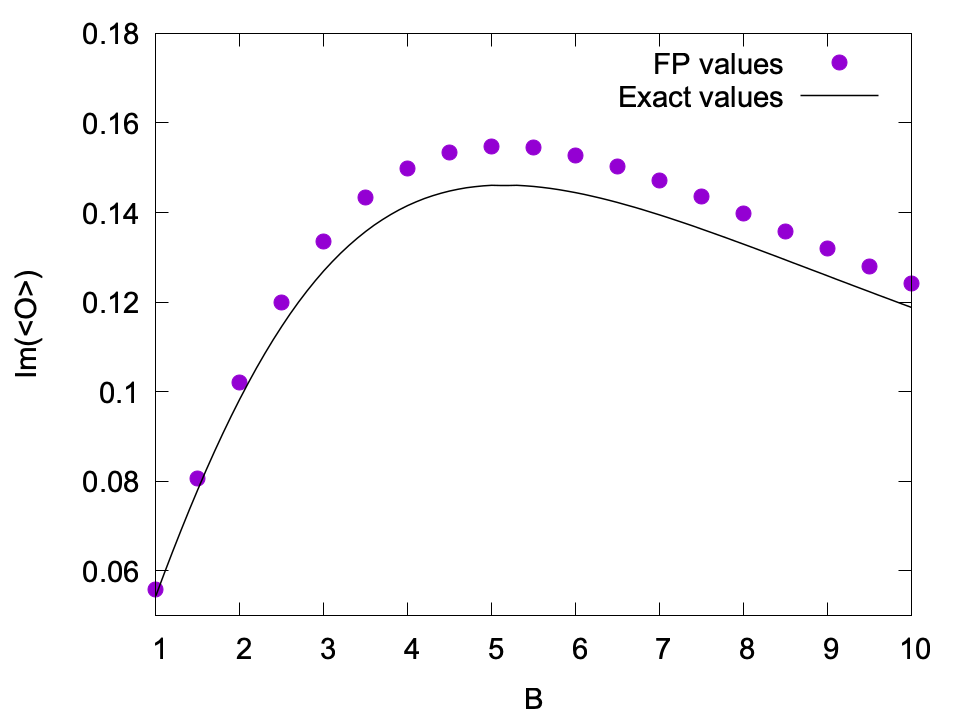}\label{fig:su2vsA5}}
  \caption{The imaginary part of the observables for $A=1$ and $A=5$. \label{fig:su2vs} }
\end{figure}

In conclusion, for the one-dimensional $SU(n)$ theory, gauge cooling can localize the probability density function for certain parameters, which stabilizes the CL method in some cases. However, when the parameters are set so that the velocity pushing samples away from the unitary field is large, the non-vanishing boundary term may still create bias in the numerical result.

\begin{remark}
In the multi-dimensional case, very similar phenomenon can be observed. In \cite{scherzer2020controlling}, the boundary term of the heavy-dense QCD is studied, which includes two parameters $\beta$ and $\mu$, denoting the gauge coupling and the chemical potential, respectively. These parameters have the similar roles to $A$ and $B$ in the above example. By numerical experiments, it is demonstrated in \cite{scherzer2020controlling} that when $\beta$ is large, although the tail of the probability density shrinks, the bias persists. The results in \cite{aarts2016qcd} for the same example show that smaller $\mu$ leads to smaller values of $F(\{U\})$ and more compact distributions of the samples, which also agrees with our observation in \cref{fig:SU2PyB}. According to the theoretical study in the one-dimensional case, we also expect a range of $\beta$ and $\mu$ where the CL dynamics with gauge cooling generates correct results, which need to be further explored in future works.
\end{remark}

\section{Conclusion and future works} \label{sec:conclusion}
This paper is devoted to the underlying mechanism of the CL method and the gauge cooling technique. By studying a controversial one-dimensional example, we have made a conclusion on the validity of the CL method with respect to its parameters. Meanwhile, it is demonstrated by this example that the use of the CL method needs to be extremely careful even for simple cases. {As pointed out in \cite{aarts2010complex,scherzer2019complex},} the decay rate of the probability density function must be carefully monitored in the numerical simulation. {A numerical approach to monitoring this decay has been proposed in \cite{scherzer2020controlling}.} The only situation in which the method can provide unbiased result for any observables is that the probability density function is localized. This occurs when all the velocities on the boundary of a certain bounded domain point inward, which may appear when the parameter controlling the magnitude of the imaginary part of the action is small. When the parameter exceeds a certain threshold, the bias in the result may arise smoothly, making it difficult to identify the appearance of the numerical failure.

For the lattice field theory, the situation is even worse since the localized probability density function does not exist. This reveals the importance of the gauge cooling technique, which introduces additional velocity that points towards the unitary field. When gauge cooling is applied, the probability density function may again be localized for certain parameters, so that the method is applicable for any observables. However, some limitations of gauge cooling, including its inability for Abelian groups and its failure in essentially suppressing the tails, is also uncovered by theoretical analysis.

This work also provides possible ideas to further develop the complex Langevin method, especially on the improvement of the dynamical stabilization proposed in \cite{attanasio2019dynamical}. This method regularizes the CL method by artificially introducing an velocity that pulls the samples back to the unitary field. This work may shed some light on the selection of the additional velocity, which should either create a velocity field that localizes the probability density function, or essentially suppresses the tail. This will be considered in our future works.
\appendix

\section{Proof of \cref{thm:CL}} \label{sec:proof}
\begin{proof}
By condition \ref{itm:H1}, we have
\begin{equation*}
\frac{\partial}{\partial t} \left( \frac{\partial \mathcal{O}}{\partial y} - i \frac{\partial \mathcal{O}}{\partial x} \right)
= \frac{\partial^2}{\partial x^2} \left( \frac{\partial \mathcal{O}}{\partial y} - i \frac{\partial \mathcal{O}}{\partial x} \right) + K_x \frac{\partial}{\partial x} \left( \frac{\partial \mathcal{O}}{\partial y} - i \frac{\partial \mathcal{O}}{\partial x} \right) + K_y \frac{\partial}{\partial y} \left( \frac{\partial \mathcal{O}}{\partial y} - i \frac{\partial \mathcal{O}}{\partial x} \right).
\end{equation*}
By the uniqueness of the solution of the advection-diffusion equation\cite{evans1999partial}, we conclude that $\partial_y \mathcal{O} = i\partial_x \mathcal{O}$ for any $x,y$ and $t$ since the initial value $\mathcal{O}(x,y;0) = O(x+iy)$ satisfies the Cauchy-Riemann equations.

For $\tau \in [0,t]$, define
\begin{equation} \label{eq:F}
F(t,\tau) = \int_{\mathbb{R}}
  \rho(x;t-\tau) \mathcal{O}(x,0;\tau) \,dx.
\end{equation}
We would like to show that $F(t,\tau)$ is independent of $\tau$, which can be done by calculating the partial derivative:
\begin{equation*}
\begin{split}
\frac{\partial}{\partial \tau} F(t,\tau) &=
  \int_{\mathbb{R}} \left[\rho(x;t-\tau) \frac{\partial}{\partial t} \mathcal{O}(x,0;\tau) - \mathcal{O}(x,0;\tau) \frac{\partial}{\partial t} \rho(x;t-\tau) \right] dx \\
&= \int_{\mathbb{R}} \rho(x;t-\tau) \left(
  \frac{\partial^2}{\partial x^2} \mathcal{O}(x,0;\tau) +
  K_x(x,0) \frac{\partial}{\partial x} \mathcal{O}(x,0;\tau) +
  K_y(x,0) \frac{\partial}{\partial y} \mathcal{O}(x,0;\tau)
  \right) dx \\
& \quad - \int_{\mathbb{R}} \mathcal{O}(x,0;\tau) \left(
  \frac{\partial^2}{\partial x^2} \rho(x;t-\tau) +
  \frac{\partial}{\partial x} [S'(x) \rho(x;t-\tau)]
  \right) dx.
\end{split}
\end{equation*}
The holomorphism of $\mathcal{O}$ indicates $\partial_y \mathcal{O} = i \partial_x \mathcal{O}$, inserting which into the above equation yields
\begin{equation*}
\begin{split}
\frac{\partial}{\partial \tau} F(t,\tau)
&= \int_{\mathbb{R}} \rho(x;t-\tau) \left(
  \frac{\partial^2}{\partial x^2} \mathcal{O}(x,0;\tau) +
  S'(x) \frac{\partial}{\partial x} \mathcal{O}(x,0;\tau)
  \right) dx \\
& \quad - \int_{\mathbb{R}} \mathcal{O}(x,0;\tau) \left(
  \frac{\partial^2}{\partial x^2} \rho(x;t-\tau) +
  \frac{\partial}{\partial x} [S'(x) \rho(x;t-\tau)]
  \right) dx.
\end{split}
\end{equation*}
According to \ref{itm:H3}, we can apply integration by parts to the above result and conclude that $\partial_{\tau} F(t,\tau) = 0$. Therefore $F(t,0) = F(t,t)$, meaning that
\begin{displaymath}
\int_{\mathbb{R}} \rho(x;t) \mathcal{O}(x,0;0) \,dx =
  \int_{\mathbb{R}} \rho(x;0) \mathcal{O}(x,0;t) \,dx.
\end{displaymath}
Applying the initial conditions of $\rho$ and $\mathcal{O}$, we get
\begin{displaymath}
\int_{\mathbb{R}} \rho(x;t) O(x) \,dx =
  \int_{\mathbb{R}} p(x) \mathcal{O}(x,0;t) \,dx.
\end{displaymath}
By comparing this equation with \eqref{eq:obs_eq} and \eqref{eq:OP}, we see that it remains only to show
\begin{equation} \label{eq:pO}
\int_{\mathbb{R}} p(x) \mathcal{O}(x,0;t) \,dx =
  \int_{\mathbb{R}} \int_{\mathbb{R}} \mathcal{O}(x,y;\tau) P(x,y;t-\tau) \,dx \,dy
\end{equation}
for some $\tau$. This can be done by setting $\tau = t$, so that the right-hand side of the above equation becomes
\begin{equation*}
\int_{\mathbb{R}} \int_{\mathbb{R}} \mathcal{O}(x,y;t) P(x,y;0) \,dx \,dy = \int_{\mathbb{R}} \int_{\mathbb{R}} \mathcal{O}(x,y;t) p(x) \delta(y) \,dx \,dy = \int_{\mathbb{R}} \mathcal{O}(x,0;t) p(x) \,dx,
\end{equation*}
which is clearly identical to the left-hand side of \eqref{eq:pO}.
\end{proof}

\section{Proof of \cref{prop:real_imag}} \label{sec:proof2}
\begin{proof}
Given \eqref{eq:VW_tau}, we can compute $\frac{dU_{\tau}}{d\tau}$ by
\begin{equation} \label{eq:dU_dtau}
\begin{split}
\frac{dU_{\tau}}{d\tau} &= (dL_{V_\tau})_{W_{\tau}} \left( \frac{dW_{\tau}}{d\tau} \right) + (dR_{W_\tau})_{V_{\tau}} \left( \frac{dV_{\tau}}{d\tau} \right) \\
& = (dL_{V_\tau})_{W_{\tau}} \Big( (dR_{W_\tau})_e (Y_\tau) \Big) + (dR_{W_\tau})_{V_{\tau}} \Big( (dR_{V_\tau})_e (X_\tau) \Big) \\
& = (dR_{W_\tau})_{V_{\tau}} \Big( (dL_{V_\tau})_e (Y_\tau) \Big) + (dR_{U_\tau})_e (X_\tau)
\end{split}
\end{equation}
Using $L_{V_{\tau}} = R_{V_{\tau}} \circ \Psi_{V_{\tau}}$, one sees that $(dL_{V_\tau})_e (Y_\tau) = (dR_{V_{\tau}})_e \big( (d\Psi_{V_{\tau}})_e (Y_{\tau}) \big)$, which can be inserted into \eqref{eq:dU_dtau} and yield \eqref{eq:U_tau} by $(dR_{W_\tau})_{V_{\tau}} \circ (dR_{V_{\tau}})_e = (dR_{U_\tau})_e$. Conversely, if \eqref{eq:U_tau} is given, then \eqref{eq:VW_tau} holds due to the uniqueness of the Cartan decomposition.

To show \eqref{eq:normal}, we write $n_\tau \in T_{W_r} \exp(i\mathfrak{g})$ as $n_{\tau} = (dR_{W_\tau})_e (n_{\tau}^e)$ for some $n_{\tau}^e \in i\mathfrak{g}$. Then by the right translational invariance of the inner product, it can be shown that
\begin{align}
\left\langle \frac{d W_{\tau}}{d\tau}, n_{\tau} \right\rangle_{U_{\tau}} &= \langle Y_{\tau}, n_{\tau}^e \rangle_e, \label{eq:ip1} \\
\begin{split}
\left\langle \frac{d U_{\tau}}{d\tau}, (dL_{V_{\tau}})_{W_{\tau}}(n_{\tau}) \right\rangle_{U_{\tau}} &= \langle X_{\tau} + (d\Psi_{V_\tau})_e (Y_{\tau}), (d\Psi_{V_\tau})_e (n_{\tau}^e) \rangle_e \\
&= \langle (d\Psi_{V_\tau})_e (Y_{\tau}), (d\Psi_{V_\tau})_e (n_{\tau}^e) \rangle_e
= \langle (dL_{V_\tau})_e (Y_{\tau}), (dL_{V_\tau})_e (n_{\tau}^e) \rangle_e
\end{split} \label{eq:ip2}
\end{align}
where the second equality of \eqref{eq:ip2} uses the fact that $(d\Psi_{V_\tau})_e$ maps $i\mathfrak{g}$ to $i\mathfrak{g}$. Note that the metric is bi-invariant on $G$, meaning that the left translation $(dL_{V_{\tau}})_e$ does not change the value of the inner product due to $V_{\tau} \in G$. Therefore \eqref{eq:ip1} and \eqref{eq:ip2} are equal, as completes the proof.
\end{proof}

\bibliographystyle{siamplain}
\bibliography{CLStrip.bib}

\end{document}